\documentclass[12pt]{article}
\usepackage{latexsym,amsfonts,amssymb,amsthm,amsmath,graphics,epsfig,extarrows}
\usepackage{comment} 
\usepackage{cancel}  
\usepackage[T1]{fontenc}
\usepackage{tikz}

\usepackage{todonotes}

\usepackage{enumitem}  

\usetikzlibrary{arrows,shapes}
\usetikzlibrary{decorations.pathmorphing}
\usetikzlibrary{calc}

\usepackage[margin=.6in]{geometry}

\usepackage[title,titletoc]{appendix}

\usepackage{empheq}
\usepackage{cases}

\usepackage{sectsty}
\sectionfont{\large}

\newcommand{\D}{\mathcal{D}}

\newcommand{\A}{\mathcal{A}}
\newcommand{\E}{\mathcal{E}}

\newcommand{\Mar}[1]{{\bf{(Marcelo: }{\color{blue} #1}})}

\newcommand{\n}[1]{\mathds{#1}}
\newcommand{\p}{A^{1/2}}

\newtheorem{theorem}{Theorem}[section]
\newtheorem{lemma}[theorem]{Lemma}
\newtheorem{proposition}[theorem]{Proposition}
\newtheorem{corollary}[theorem]{Corollary}

\newtheorem{assumption}{Assumption}

\theoremstyle{remark}
\newtheorem{remark}{Remark}[section]

\numberwithin{equation}{section}
\numberwithin{theorem}{section}

\usepackage{everysel}
\EverySelectfont{%
	\fontdimen2\font=0.4em
	\fontdimen3\font=0.2em
	\fontdimen4\font=0.1em
	\fontdimen7\font=0.1em
	\hyphenchar\font=`\-
}
\usepackage{tgcursor}
\usepackage{dsfont}
\usepackage{cases}
\usepackage[colorlinks=true]{hyperref}
\hypersetup{urlcolor=blue, citecolor=blue}

\newcommand{\cA}{{\mathcal{A}}}
\newcommand{\cF}{{\mathcal{F}}}

\newcommand{\calD}{{\mathcal{D}}}

\newcommand{\calL}{{\mathcal{L}}}

\newcommand{\dvv}[1]{{\mbox{div}(#1)}}

\newcommand{\BR}{\mathbb{R}}

\begin{document}
\title{Boundary feedback  stabilization of a critical  nonlinear JMGT equation  with Neumann-undissipated  part of the boundary
}
	
	\author{
		 Marcelo Bongarti$^1$ \and Irena Lasiecka$^2$
	}
	\date{\small %
	$^1$Weierstrass Institute for Applied Analysis and Stochastics, Berlin, Germany\\%
	$^2$Department of Mathematical Sciences, University of Memphis\\[2ex]%
	\today
}
	\maketitle
	\begin{abstract}
	
	 Boundary feedback stabilization of a critical, nonlinear Jordan--Moore--Gibson--Thompson (JMGT) equation  is considered. JMGT arises in modeling of acoustic waves involved in medical/engineering  treatments like lithotripsy, thermotherapy, sonochemistry, or any other procedures using High Intensity Focused Ultrasound (HIFU).  It is a well-established and recently widely studied model for nonlinear acoustics (NLA): a third--order (in time) semilinear Partial Differential Equation (PDE) with the distinctive feature of predicting the propagation of ultrasound waves at \textit{finite} speed due to heat phenomenon know as \textit{second sound}  which leads to the hyperbolic character of heat propagation.
 In practice, the JMGT dynamics is largely used for modeling the evolution of the acoustic velocity and, most importantly, the acoustic pressure as sound waves propagate through certain media. 
 In this work, \emph{critical} refers to (usual) case where media--damping effects are non--existent or non--measurable and therefore cannot be relied upon for stabilization purposes. 
  
	  In this paper the issue of  boundary stabilizability of originally  unstable  (JMGT) equation is resolved. Motivated by modeling aspects  in HIFU technology, boundary feedback  is supported only on a portion of the boundary, while the remaining part of the boundary is left free (available to  control actions) .  Since the boundary conditions imposed on the  "free" part of the boundary  fail to satisfy  Lopatinski condition (unlike Dirichlet boundary conditions), the analysis of  uniform stabilization from the boundary becomes very subtle and requires careful geometric considerations. 
	 
	\end{abstract}
	\noindent{\bf keywords: } Nonlinear acoustics, second sound, third--order in time, heat-conduction, boundary stabilization, degenerate viscoelasticity.
	

\section{Introduction} \subsection{PDE Model and an Overview}

Let $\Omega \subset \BR^d$ $(d = 2, 3)$ denote a bounded domain with sufficiently smooth boundary $\Gamma := \partial \Omega$ within which a sound wave propagates. In HIFU technology, as well as in other contexts, one is interested in tracking -- and often controlling -- the evolution of the acoustic pressure $u = u(t,x)$ ($t \in \mathbb{R}_+, x \in \Omega)$ caused by such wave  propagation. In media on which heat propagates hyperbolically (which is the case of most biological tissues), the evolution of the acoustic pressure can be assumed to obey the semilinear JMGT--equation which is given by the {\it third order in time abstract evolution}  equation \begin{equation}
    \label{modnl} \tau u_{ttt} + (\alpha - 2ku)u_{tt} - c^2 \Delta u - (\delta + \tau c^2)\Delta u_t = 2ku_t^2,
\end{equation} where $c, \delta, k > 0$ are constants representing the speed and diffusivity of sound and a nonlinearity parameter, respectively, while the function $\alpha: \overline\Omega \to \BR^+$ accounts for natural frictional damping provided by the media. The parameter $\tau >0$ -- also media--dependent -- accounts for thermal relaxation and essentially transfers the hyperbolicity of the heat to the acoustic wave. 

The semilinear equation \eqref{modnl} can be viewed as a singular perturbation and, to some extent, a \textit{refinement}  of the classical quasilinear Westervelt's equation \begin{equation}
    \label{modnl2} (\alpha - 2ku)u_{tt} - c^2 \Delta u - \delta\Delta u_t = 2ku_t^2.
\end{equation} obtained by setting $\tau = 0.$ Physically, the main difference between \eqref{modnl} and \eqref{modnl2} is that the latter predicts that these waves propagate at a  finite  speed.   From the modeling point of view this results from using  Maxwell-Cattaneo Law -- rather than Fourrier's Law -- to model a heat flux  for acoustic heat waves.  The parameter $\tau >0 $ corresponds to time relaxation. For more about the physical interpretation of \eqref{modnl}, its derivation and overall discussion see \cite{bongarti_lasiecka_rodrigues_DCDS_2021,ekoue_halloy_gigon_plantamp_zajdman_ETIJPMS_2013,cattaneo_CR_1958,cattaneo_SBH_2011,spigler_MMAS_2020,christov_jordan_PRL_2005,kaltenbacher}.This also includes an analysis of asymptotic behavior of solutions when the parameter of relaxation tends to zero \cite{kaltenbacher_nikolic_MMMAS_2019,kaltenbacher_nikolic_PAFA_2019,bongarti_charoenphon_lasiecka_JEE_2021,bongarti_charoenphon_lasiecka_SOTA_2020}

The issues of wellposedness and stability of solutions under  homogeneous Dirichlet and Neumann boundary data were first addressed for both nonlinear and linearized ($k = 0$) dynamics around 2010 with the works of I. Lasiecka, R. Triggiani and B. Kaltenbacher \cite{kaltenbacher_lasiecka_PJM_2012,kaltenbacher_lasiecka_marchand_CC_2011,marchand_triggiani_MMA_2012,kaltenbacher_lasiecka_pos}. For the analysis of the long--time dynamics of \eqref{modnl} for both linear and nonlinear cases, the function \begin{equation}\label{gm}\gamma:\overline\Omega \to \BR, \qquad \gamma(x) \equiv \alpha(x) - \frac{\tau c^2}{b}\end{equation} plays a central role. In fact, the existence of a positive constant $\gamma_0$ such that $\gamma(x) \geqslant \gamma_0 >0$ a.e. in $\Omega$ ensures both that linear dynamics is uniformly exponentially stable and that stable nonlinear flows can be constructed via ``barriers'' \cite{kaltenbacher_lasiecka_pos} method. A natural question, in light of the above  results, is concerning  the other profiles of $\gamma.$ It is known that if $\gamma < 0 $ one may have chaotic solutions \cite{conejero_lizama_rodenas_AMIS_2015} and if $\gamma \equiv 0$ then the energy is conserved \cite{kaltenbacher_lasiecka_marchand_CC_2011,kaltenbacher_lasiecka_PJM_2012}. This raises the interesting question of what mechanisms  could be employed to ensure stability of the dynamics when $\gamma $ degenerates, i.e., $\gamma(x)  \geqslant 0 $. 

	From a practical point of view, the quantity $\gamma(x)$ is interpreted as the viscoelasticity of the material point $x \in \Omega$ and, in particular in the medical field, is not expected to be known for all points of $\Omega$. By making the physically relevant assumption that $\gamma \in L^\infty(\Omega)$,  $\gamma(x) \geqslant 0$ a.e. in $\Omega$ (allowing the critical case $\gamma \equiv 0$, or the case where measurements can only me made at isolated points of the domain), we ask ourselves whether a non--invasive (boundary) action can drive the acoustic pressure to zero at large times regardless of the particular knowledge of $\gamma$ (as long as it is nonnegative). This question, besides being of independent interest in stability theory, is critical in ensuring {\it global} wellposedness of nonlinear solutions. Otherwise  the  nonlinearity may cause  "blow" up of solutions \cite{chen_palmieri_EECT_2019}.
	
It has been recently shown that viscoelastic effects produces, in some cases,  the asymptotic decay of the energy, cf. e.g. \cite{lasiecka_wang_JDE_2015,lasiecka_wang_JAMP_2016,delloro_lasiecka_pata_JDE_2016,delloro_lasiecka_pata_JEE_2020,delloro_pata_AMO_2017,delloro_pata__MJM_2017}. In this work we concentrate on  a physically attractive boundary stabilization -- where control action can be applied  on the boundary, hence easily accessible to external manipulations. 
Of particular interest is a configuration arising in the  ultrasound technology where  an acoustic medium is excited  on one part of the boundary, while the remaining part of the boundary is subject to absorbing boundary conditions.This control model was introduced in \cite{clayson,clayson1} in the case of Westervelt-Kuznetsov equation and later pursued in   \cite{bucci_lasiecka_O_2019} for MGT equation.
This corresponds to  the following boundary conditions \begin{equation}
    \label{BC1} 	\lambda \partial_\nu u + \kappa_0(x)u = 0 \ \mbox{on} \  \Sigma_0 \qquad \partial_\nu u + \kappa_1(x)u_t = 0 \ \mbox{on} \ \Sigma_1
\end{equation}
 with  $\Gamma_0, \Gamma_1 \subset \Gamma $ relatively open, $\Gamma_0 \neq \emptyset$, $\overline{\Gamma_0} \cup \overline{\Gamma_1} = \Gamma, \Gamma_0 \cap \Gamma_1 = \emptyset$, $\lambda > 0$, $\kappa_0 \in L^\infty(\Gamma_0)$ and $\kappa_1 \in L^\infty(\Gamma_1),$ $\kappa_1(x) \geq \kappa_1 > 0 , \kappa_0 >0$ a.e.
 
 Notice that the boundary condition \eqref{BC1}$_1$ -- where there is {\it no dissipation --, do not satisfy strong  Lopatinski condition}, a  fact that leads to  new challenges at the level of proving controllability or stabilization even for  a wave equation in dimension higher than one. The technical (mathematical)  reason is that the  presence of tangential boundary derivatives cannot be handled by standard flow  multipliers methods.   In fact, past contributions to the subject include \cite{bongarti_lasiecka_DSOCIP_2021,bongarti_lasiecka_rodrigues_DCDS_2021} where   {\it linear}  dynamics  is considered in the case $\lambda = 0$ and $\kappa_0 \equiv \kappa_1 \equiv 1$ in \eqref{BC1}.  Thus, the uncontrolled part of the boundary is subject to {\it Dirichlet } boundary conditions where Lopatinski condition is satisfied and tangential derivatives (appearing in the applications of flux multipliers) vanish altogether on $\Gamma_0.$ In \cite{bongarti_lasiecka_DSOCIP_2021} star--shaped boundary condition is assumed on $\Gamma_1$. This restriction has been removed in \cite{bongarti_lasiecka_triggiani_AA_2021} by resorting to a  microlocal analysis argument. 
 
	 The present paper  addresses   the challenging case of  a {\it nonlinear, critical}  JMGT dynamics subject to  {\it Neuman/Robin boundary conditions on the undissipated part of the boundary $\Gamma_0$} ($\lambda =1$). The  model is important in the context of HIFU control theory where boundary  {\it open loop} strategic control is activated precisely on  this ``free'' part of the boundary. On the other hand, in such case one encounters  a well recognized PDE predicament: seeking stabilization  for a  hyperbolic dynamics when Lopatinski condition \cite{sakamoto_CUP_1982,sakamoto_CUP_2009} fails on  undissipated  portion of the boundary.  This leads to major difficulties when applying flux multipliers or geometric optics in order to carry  the analysis of uniform  stability. 
	    Clearly, one expects  some  restrictions on the geometry of the boundary to cooperate. In the case of Dirichlet boundary conditions, star--shaped condition suffices. 
	    Instead, for the Neumann case (non--Lopatinski),  it turns out that star--shaped along with  some {\it convexity} is a sufficient condition. Precise formulation of the corresponding results will be given in the next section. In addition to new geometric constructs,   nonlinearity in the model forces  considering stability properties  at higher topological levels with a restricted ``smallness'' condition   imposed on  the initial data. The key point  here is that this would make the model ``close'' to linear. 
	    In order to contend with this  limitation,  the results presented require smallness conditions imposed only at the low energy level, while higher derivatives can remain large. We will be able to achieve this  goal through  boundary dissipation and small initial data  imposed {\it only}  at the lowest energy level. So the model and the resulting acoustic waves  remain genuinely nonlinear. 
	    
	    For other  relatively  recent references related to  regularity questions  for linear MGT equation, an  interested reader  may be referred to : \cite{bucci_eller_CRM_2021,pellicer_solamorales_EECT_2019,triggiani_SOTA_2020} 

	\section{Main Results and Discussion.} \label{secwpp}
	We consider the system comprised of \eqref{modnl},  boundary conditions \eqref{BC1} and initial conditions \begin{equation}
	    \label{IC1} u(0,\cdot) = u_0, \qquad u_t(0,\cdot) = u_1, \qquad u_{tt}(0,\cdot) = u_2
	\end{equation} with regularity to be specified in what follows. 
	
	 Here and throughout the paper, by $L^2(\Omega)$ and $L^2(\Gamma)$ we denote the sets of measurable (in the Lebesgue and Hausdorff senses, respectively) functions whose squares are integrable on $\Omega$ and $\Gamma$ respectively equipped with the norms induced by the inner products $$(u,v)=\int_\Omega uv d\Omega \ \ \mbox{and} \ \  (u,v)_\Gamma=\int_\Gamma uv d\Gamma.$$ and  denoted respectively by $\|\cdot\|_2$ and $\|\cdot\|_\Gamma$. The remaining $L^p(\Omega)$--spaces ($1 \leqslant p \leqslant \infty$) will also have their norms denoted by $\|\cdot\|_p$. Additionally,  by $H^s(\Omega)$ we denote the ($L^2$--based) Sobolev space of order $s$ and define the particular spaces $H_\Gamma^1(\Omega)$ and $H_\Gamma^2(\Omega)$ as $$H_\Gamma^1(\Omega) = \left\{u \in H^1(\Omega) ; \ u|_\Gamma = 0\right\} \ \mbox{and} \ H_\Gamma^2(\Omega) = H^2(\Omega) \cap H_\Gamma^1(\Omega)$$ in order to avoid confusion with the standard $H_0^1(\Omega)$.
	
	\subsection{Functional Analytic Setting.} 
Let $A: \D(A) \subset L^2(\Omega) \to L^2(\Omega)$ be the operator defined as \begin{equation}\label{oplap} A\xi = -\Delta \xi,  \ \ 
\D(A) = \left\{\xi \in H^2(\Omega); \  \partial_\nu\xi\rvert_{\Gamma_1} =0, \left[\partial_\nu\xi+\kappa_0 \xi\right]_{\Gamma_0}]= 0 \right\}. 
\end{equation}In this setting, $A$ is a positive,  self--adjoint operator with compact resolvent and for $\kappa_0 > 0  $,  $\D\left(\p\right) = H^1(\Omega)$ with the -- equivalent to $H^1(\Omega)$ -- topology of $\calD(A^{1/2})$ given by $$\|u\|_{\calD(\p)}^2 := \|\nabla u\|_2^2 + \int_{\Gamma_0}\kappa_0|u|^2d\Gamma_0.$$ In addition, with some abuse of notation we (also) denote by $A: L^2(\Omega) \to [\D(A)]'$ the extension (by duality) of the operator $A.$

Let us  introduce the phase space $\mathbb{H}$ given by  \begin{equation}\label{ph-sp}\mathbb{H} := \D(A^{1/2} ) \times 
	\D(A^{1/2})  \times L^2(\Omega) \sim H^1(\Omega) \times H^1(\Omega) \times L^2(\Omega).\end{equation}
Next, we rewrite \eqref{modnl} along with \eqref{BC1} and \eqref{IC1} as a first--order abstract system on $\mathbb{H}$. For this, we introduce the classic boundary $\to$ interior harmonic extension for the Neumann data on $ \Gamma_1$ as follows: for $\varphi \in L^2(\Gamma_1)$, let $\psi := N(\varphi),$ be the unique solution of the elliptic problem \begin{equation}
\begin{cases}
\Delta\psi = 0 \ & \mbox{in} \ \Omega \\ \partial_\nu \psi = \varphi\rvert_{\Gamma_1} \ & \mbox{on} \ \Gamma_1 \\  \partial_\nu\psi +\kappa_0 \psi  = 0 \ & \mbox{on} \ \Gamma_0.
\end{cases} \label{ep}
\end{equation} From elliptic theory, it follows that that $N \in \mathcal{L}(H^s(\Gamma_1),H^{s+3/2}(\Omega))$ $(s \in \mathbb{R})$ and \begin{equation} \label{neq} N^\ast A \xi = \begin{cases}
\xi \ \mbox{on} \ &\Gamma_1 \\ 0 \ \mbox{on} \ &\Gamma_0,
\end{cases}\end{equation} for all $\xi \in \D(A)$, where $N^\ast$ represents the adjoint of $N$
when it is considered as an operator from $L^2(\Gamma_1) $ to $L^2(\Omega)$ \cite{lasiecka_triggiani_CUP_2010}. 

Thus, the $u$--problem can be written (via duality on  $[\D(A)]'$)  as \begin{align}\label{absver}
\tau u_{ttt} + \alpha u_{tt} + c^2Au + bAu_t + c^2 AN(\kappa_1N^*Au_t) + bAN(\kappa_1N^*Au_{tt}) = u_t^2+uu_{tt}
\end{align} where we have taken $k = 1/2$ without any loss of generality.

Next, we introduce the operator $\cA: \calD(\cA) \subset \mathbb{H}  \to \mathbb{H}$ with  the action:
\begin{align}\label{opuA}
\cA\begin{bmatrix}\xi_1 \\[2mm] \xi_2 \\[2mm] \xi_3\end{bmatrix} := \begingroup 
\setlength\arraycolsep{12pt} \begin{bmatrix} 0 & I & 0 \\[2mm] 0 & 0 & I \\[2mm] -\dfrac{c^2}{\tau}A & -\dfrac{c^2}{\tau}AN(\kappa_1 N^\ast A) - \dfrac{b}{\tau}A & -\dfrac{b}{\tau}AN(\kappa_1 N^\ast A) - \dfrac{\alpha}{\tau} I\end{bmatrix}\endgroup \begin{bmatrix} \xi_1 \\[2mm] \xi_2 \\[2mm] \xi_3 \end{bmatrix}\end{align} and domain  (with $\vec{\xi}\equiv (\xi_1,\xi_2,\xi_3)^\top$ \begin{align}
\label{domA} \D(\cA) &= \left\{ \vec{\xi} \in \mathbb{H};    \ \xi_3 \in \D\left(A^{1/2}\right), \ \xi_i + N(\kappa_1N^*A\xi_{i+1})\in D(A),~\mbox{for}~i=1,2\right\} \nonumber \\ &= \left\{\vec{\xi}  \in \left[H^2(\Omega)\right]^2 \times  H^1(\Omega) ; \ \left[ \partial_\nu \xi_1 + \kappa_0\xi_1\right] _{\Gamma_0}= 
	\left[\partial_\nu \xi_2 + \kappa_0\xi_2\right]_{\Gamma_0} =0 \right. \nonumber \\ & \hspace{5.9cm}\left. \left[\partial_\nu \xi_1+ \kappa_1\xi_2\right] _{\Gamma_1}= 
	\left[\partial_\nu \xi_2 + \kappa_1\xi_3\right]_{\Gamma_1} =0
\right\}
\end{align} 
where the second characterization follows from elliptic regularity. This gives $$\D(\cA) \subset H^2(\Omega) \times H^2(\Omega) \times H^1(\Omega) $$ with  a proper, but not dense injection. 


The first order abstract version of the $u$--problem is thus given by 
\begin{equation} \begin{cases}
\label{usist} \Phi_t = \cA\Phi + \cF(\Phi)\\ 
\Phi(0) = \Phi_0 = (u_0,u_1,u_2)^\top,\end{cases}
\end{equation} in the variable $\Phi = (u,u_t,u_{tt})^\top$ with $\cA$
defined in (\ref{opuA}) and $\cF(\Phi)^\top \equiv (0,0,\tau^{-1}(u_t^2+uu_{tt})).$ 

 In order to treat nonlinear problem one needs     to consider ``smoother'' solutions than generated by the topology  of $\mathbb{H}$. 
 This leads ton the following construction of the second phase space  denoted by $\mathbb{H}_1 $, which is ``thighter'' than $\mathbb{H}$ but strictly larger than  $\D(\cA) $.
    The new phase space $\mathbb{H}_1$ is 
    defined below
     \begin{equation}\label{newphase}\mathbb{H}_1 = \{ \vec{\xi} \in \mathbb{H}; \Delta \xi_1 \in L^2(\Omega);   \left[\lambda \partial_\nu \xi_1 + \kappa_0\xi_1\right] _{\Gamma_0}= 0;  \left[\partial_\nu \xi_1+ \kappa_1\xi_2\right] _{\Gamma_1}= 0
 \} 
\end{equation} 
and endowed with the norm 
$$\|\vec{\xi} \|^2_{\mathbb{H}_1} = \|\vec{\xi} \|^2_{\mathbb{H}} + \|\Delta \xi_1\|_2^2 +  \|\xi_1\|^2_{H^{1/2}(\Gamma_0)}+ \|\xi_2\|^2_{H^{1/2}(\Gamma_1)}$$ 
or equivalently 
$$\|\vec{\xi} \|^2_{\mathbb{H}_1} = \|\vec{\xi} \|^2_{\mathbb{H}} + \|\Delta \xi_1\|_2^2 +  \|\partial_{\nu} \xi_1\|^2_{H^{1/2}(\Gamma)}$$ 
Note that the  boundary conditions in the definition of  the space $\mathbb{H}_1$ are  well defined due to the property: 
$\Delta \xi_1 \in L^2(\Omega) $ and $\xi_1 \in H^1(\Omega) $ then $ \partial_{\nu} \xi_1 \in H^{-1/2}(\Gamma) $ -- the latter allowing to define the boundary conditions as a distribution. 
We  also note that  since $\xi_{1},\xi_2 \in H^1(\Omega)$ we have $\xi_{i}|_{\Gamma}  \in H^{1/2} (\Gamma)$ ($i = 1,2$) and therefore $ \partial_{\nu} \xi_1\in H^{1/2}(\Gamma)$. This along with  the elliptic regularity implies: 
$$\mathbb{H}_1 \subset H^2(\Omega) \times H^1(\Omega) \times L^2(\Omega) $$  with a proper but  {\it not dense} injection. We shall show that the operator $\cA$ also generates a $C_0$--semigroup $\{T(t)\}_{t \geqslant 0}$ on $\mathbb{H}_1.$ Notice that  the nonlinear term is invariant under $\mathbb{H}_1 $ topology in dimensions up to 3.
\subsection{Main Results} 
Our first preliminary result is stated below. \begin{theorem}\label{gen}The operator $\cA$ generates a $C_0$--semigroup $\{S(t)\}_{t \geqslant 0}$ on $\mathbb{H}$. Moreover, the family $T(t):= S(t) \rvert_{\mathbb{H}_1}$, $t \geqslant 0,$ is also a $C_0$--semigroup with generator $\cA$ and  its realization  on  $\mathbb{H}_1$.
\end{theorem}The second result deals with  an exponential stability  of the  semigroups on the phase space $\mathbb{H} $ and  $\mathbb{H}_1$. For this, one needs to introduce the following geometric condition. 
\begin{assumption}\label{geo}
The boundary $\Gamma_0$ is star--shaped and convex. This is to say: there exists $x_0 \in \BR^n $ such that $(x-x_0) \cdot \nu(x) \leqslant 0$  for all $x \in \Gamma_0$  where   $\n(x) $ is the outwards normal vector to the boundary at $x$. In addition, there exists a convex level set function which defines  $\Gamma_0.$ See \emph{\cite{lasiecka_triggiani_zhang_CM_2000}.}
\end{assumption}
\begin{theorem}[\bf Two level uniform stability] \label{thm34n}
 Let Assumption \ref{geo} on $\Gamma_0$ be in force and let $\gamma(x) \geqslant 0$. Then {\bf \emph{(i)}}
the semigroup $\{S(t)\}_{t \geqslant 0}$ generated by $\cA$ in $\mathbb{H}$ is uniformly exponentially stable with decay rate $\omega_0 > 0$ and {\bf \emph{(ii)}}
 	the semigroup  $\{T(t)\}_{t \geqslant 0}$ generated by $\cA$ in $\mathbb{H}_1$ is uniformly exponentially stable with decay rate $\omega_1>0$, where $\omega_1 < \omega_0$.

\end{theorem}

Once linear wellposedness and uniform stability  of the linear ($k=0$) problem are established with respect to  the appropriate topologies, 
 our next task  is to prove generation of nonlinear semigroup  on $\mathbb{H}_1 $. To accomplish this, initial data need to be assumed sufficiently small. \textit{How small?} This is  an important question as argued
  in \cite{bongarti_charoenphon_lasiecka_JEE_2021}. We will be able to show that  some smallness will be only   imposed at the lowest level of regularity, while higher derivatives can be large. 
  As a consequence,  in the following theorem we show existence of $\mathbb{H}_1$--valued solutions given $\mathbb{H}_1$ initial data which are small in $\mathbb{H}$ only. The proof, given in Section \ref{nonlinearsem}, relies on estimates derived via interpolation inequalities which allows to demonstrate certain  ``invariance''  of a  $\mathbb{H}$-small ball under the  nonlinear dynamics in $\mathbb{H}_1$. 

We start specifying the notion of solution for the semilinear problem \eqref{modnl} supplemented with \eqref{BC1} and \eqref{IC1}. We denote the initial data here by $\Phi_0 = (u_0,u_1,u_2)^\top.$ Given $T>0$, we say that $$\Phi(t) = (u(t),u_t(t),u_{tt}(t))$$ is a {\bf mild solution} for the system  \eqref{modnl}, \eqref{BC1} and \eqref{IC1} provided $\Phi  \in C([0,T],\mathbb{H}_1)$ and \begin{equation}
\label{formnon}\Phi(t) = T(t)\Phi_0 + \int_0^t T(t-\tau)\cF(\Phi)(\tau)d\tau,
\end{equation} 

Before stating the theorem, we denote by $\mathbb{H}^\rho$ (for $\rho > 0$) the set $$\mathbb{H}^\rho := \left\{\Phi \in \mathbb{H}_1; \|\Phi\|_{\mathbb{H}} <\rho \right\}.$$\begin{theorem}[\bf Global Solutions]\label{nsem} Let Assumption \ref{geo} on $\Gamma_0$ be in force. Then, there exists $\rho > 0$ sufficiently small such that, given any $\Phi_0 \in \mathbb{H}^\rho$ the formula \eqref{formnon} defines a continuous $\mathbb{H}_1$--valued mild solution for the system \eqref{modnl}, \eqref{BC1} and \eqref{IC1}. Moreover, for such $\rho > 0$, there exists $R = R(\|\Phi_0\|_{\mathbb{H}_1})$ such that all trajectories starting in $B_{\mathbb{H}^\rho}(0,R)$\footnote{The $\mathbb{H}^\rho$--ball centered at the origin and with radius $R.$} remain in $B_{\mathbb{H}^\rho}(0,R_1)$ for all $t \geqslant 0, R_1 > R. $\end{theorem}

Once global solutions are shown to exist, we take on the issue of asymptotic (in time) stability. 
The final result is positive, as we expected, and holds uniformly (w.r.t $\gamma$) as long as $\gamma \in L^\infty(\Omega)$ and $\gamma(x) \geqslant 0$ a.e. in $\Omega$. \begin{theorem}[\bf Nonlinear Uniform Stability]\label{expnonl}
Let Assumption \ref{geo} on $\Gamma_0$ be in force and assume $\gamma \in L^\infty(\Omega)$ and $\gamma(x) \geqslant 0.$ Then, there exists $\rho>0$ sufficiently small and $M(\rho) , \omega > 0$ such that if $\Phi_0 \in \mathbb{H}^\rho$ then \begin{equation}
	\label{expdec} \|\Phi(t)\|_{\mathbb{H}_1} \leqslant M(\rho) e^{-\omega t}\|\Phi_0\|_{\mathbb{H}_1}, \qquad t \geqslant 0
	\end{equation} where $\Phi$ is the mild solution given by Theorem \ref{nsem}.
\end{theorem} 
\ifdefined\xxxxx
\begin{corollary}\label{corexp} With reference to Section 5, let $\beta_0$ be the biggest number such that the map $\Theta$ has a fixed point in $X^{\beta_0}$ which is, moreover, uniformly exponentially stable as in Theorem \ref{expnonl}. Let $\omega: (0,\beta_0] \to \BR_+$ be the function that maps each $\beta>0$ to the decay rate $\omega(\beta)$. Then there exists another function $\underline \omega: (0,\beta_0] \to \BR_+$ such that $\omega(\beta) \geqslant \underline\omega(\beta)$ for all feasible $\beta$ and \begin{equation}
	\label{limexp} \lim\limits_{\beta \to 0} \underline\omega(\beta) = \omega_1,
	\end{equation} where $\omega_1$ is the decay rate of the linear semigroup $T(t).$
\end{corollary}
\fi
We notice that repeating the statement \textit{there exists $\rho >0$ such that if $\Phi_0 \in \mathbb{H}^\rho$} in the Theorem above is not redundant. In fact, one might need to require even smaller initial data to yield exponential decay. This again highlights the advantage of requiring smallness only in $\mathbb{H}$. For more details, see Section \ref{refext2}.
\subsection{Discussion}
 The main novelty of this paper is that we study stabilizability of a  {\it nonlinear} {\it  critical} JMGT equation with  Neumann--Robin undissipated portion of the boundary. 
 Should the problem be subcritical (i.e. $\gamma(x)  > \gamma_0 > 0$  for  $ x\in \Omega$, the difficulty created  by the  failure of Lopatinski condition would not enter the picture. Simply because  there will be no need to propagate stability from the boundary into the interior.  As already mentioned before, linear dynamics with absorbing boundary conditions on $\Gamma_1$ and zero Dirichlet data on $\Gamma_0$  subject to star--shaped conditions has  been considered in \cite{bongarti_lasiecka_DSOCIP_2021,bongarti_lasiecka_triggiani_AA_2021}. Mathematical difficulties in propagating stability through  the undissipated part of the boundary are  not present in this case. In order to cope with the difficulties we shall employ geometric constructs developed earlier  in \cite{lasiecka_triggiani_zhang_CM_2000}. These allow to construct suitable--non--radial--vector fields  which result from tangential bending of radial and star--shaped ones.  These newly constructed fields  propagate the needed estimates through un--dissipated part of the boundary. 

In order to treat  the nonlinear problem,  the  approach used in the past (for subcritical case)  was to use the so called  ``barrier's method'' based on contradiction argument. However, this presents  several technical  difficulties in the present scenario,  even at the level of  low frequencies (lower order terms).  Hence, in this paper,  we exploit another technique which, to the best of our knowledge, is new and makes a strong use of the fact that we \emph{only} require initial data to be small in $\mathbb{H}.$ One of the advantages of such construction (for JMGT) was already exploited by the authors in \cite{bongarti_charoenphon_lasiecka_JEE_2021} in allowing extension by density in the nonlinear environment. In this paper we discovered that it also allows  to: \begin{itemize}
		\item[\bf a)] prove  global existence and exponential stability by the representation of the solution and  two-level stability of linear flows. Here, the smallness interplay comes to the  picture through  a nonlinear propagation of the  estimate of the type \begin{equation}\|\cF(\Phi)\|_{\mathbb{H}_1} \leqslant C_1(\|\Phi\|_{\mathbb{H}} )C_2(|\Phi\|_{\mathbb{H}_1}.)\end{equation}
		where the size of  $C_1(\|\Phi\|_{\mathbb{H}})$ can be controlled by  $\rho$. See Theorem \ref{nsem}.
		\item[\bf b)] obtain, to some extent, a continuity property of the decay rate with respect to the $\mathbb{H}$--size of the initial data and the decay rate of the linear flow, $\omega_1$. In general, we prove that if $\varepsilon$ is the $\mathbb{H}$--size of the initial data and $\omega(\varepsilon)$ is the corresponding decay rate, then there exists $\underline{\omega}(\varepsilon)$ such that $\omega(\varepsilon) \geqslant \underline\omega(\varepsilon)$ and $\underline\omega(\varepsilon) \to \omega_1$ as $\varepsilon \to 0^+.$ 
	\end{itemize} 
 
\ifdefined\xxxxxx
The approach we are going to use to study the boundary stabilization problem consists of the following. \begin{itemize}
    \item[\bf (i)] We  will first show that the operator $\cA$ generates a $C_0$--semigroup $\{S(t)\}_{t \geqslant 0}$ on $\mathbb{H}$. As it is well known, it is enough to get $\mathbb{H}$--valued continuous solutions of the first order system $\Phi_t = \cA\Phi + G$ as long as the forcing term $G$ belongs to $L^1(\BR_+,\mathbb{H})$ and initial data \eqref{IC1} are assumed in $\mathbb{H}.$ Such solutions can be represented by the variation of parameter formula, i.e., $$\Phi(t) = S(t)\Phi_0 + \int_0^t S(t-\sigma)G(\sigma)d\sigma.$$ Moreover, they become classical under the assumption of smoothness of the data. Since the nonlinear terms generate elements from the outside of the main phase space $\mathbb{H} $, an additional step of analyzing linear wellposedness in a "tighter space"  will be necessary as  explained below.
    The "tighter" space $\mathbb{H}_1$ is between  $\mathbb{H}$ and $\D(\cA) $. 
    
    \item[\bf (ii)] Second, we use the representation of solutions as given above to construct fixed point solutions for the semilinear problem. 
    However, notice that $\cF(\Phi)$ does \emph{not} necessarily belong to $\mathbb{H}$ for $\Phi \in \mathbb{H}$. This indicates a need for  a "tighter" space which is also invariant under the dynamics.
    To this end a new phase space $\mathbb{H}_1$ is introduced and defined below.
     \begin{equation}\label{newphase}\mathbb{H}_1 = \{ \vec{\xi} \in \mathbb{H}; \Delta \xi_1 \in L_2(\Omega);   \left[\lambda \partial_\nu \xi_1 + \kappa_0\xi_1\right] _{\Gamma_0}= 0;  \left[\partial_\nu \xi_1+ \kappa_1\xi_2\right] _{\Gamma_1}= 0
 \} 
\end{equation} 
With the norm 
$$\|\vec{\xi} \|^2_{\mathbb{H}_1} = \|\vec{\xi} \|^2_{\mathbb{H}} + \|\Delta \xi_1\|_2^2 +  \|\xi_1\|^2_{H^{1/2}(\Gamma_0)}+ \|\xi_2\|^2_{H^{1/2}(\Gamma_1)}$$ 
or equivalently 
$$\|\vec{\xi} \|^2_{\mathbb{H}_1} = \|\vec{\xi} \|^2_{\mathbb{H}} + \|\Delta \xi_1\|_2^2 +  \|\partial_{\nu} \xi_1\|^2_{H^{1/2}(\Gamma)}$$ 
Note that the  boundary conditions in the definition of  the space $\mathbb{H}_1$ are  well defined due to the property: 
$\Delta \xi_1 \in L_2(\Omega) $ and $\xi_1 \in H^1(\Omega) $ then $ \partial_{\nu} \xi_1 \in H^{-1/2}(\Gamma) $ -- the latter allowing to define the boundary conditions. 
We  also note that  since $\xi_{1},\xi_2 \in H^1(\Omega)$ we have $\xi_{i}|_{\Gamma}  \in H^{1/2} (\Gamma)$ ($i = 1,2$) and therefore $ \partial_{\nu} \xi_1\in H^{1/2}(\Gamma)$. This along with elliptic regularity implies: 
$$\mathbb{H}_1 \subset H^2(\Omega) \times H^1(\Omega) \times L_2(\Omega) $$  with a proper but  {\it not dense} injection. We shall show that the operator $\cA$ also generates a $C_0$--semigroup $\{T(t)\}_{t \geqslant 0}$ on $\mathbb{H}_1.$ Notice that  the nonlinear term is invariant under $\mathbb{H}_1 $ topology in dimensions 1-3. 
    
    \item[\bf (iii)] The above allows for a construction of local in time solutions to the nonlinear problem.  Once the  existence and uniqueness of $\mathbb{H}_1$--valued continuous solutions are obtained, we extend the long time behavior  estimates -- already available for $\mathbb{H}$--solutions -- to the $\mathbb{H}_1$--solutions.
    This will lead  to uniform exponential stability  of linear problem in $\mathbb{H}_1$ topology. 
    This result,  beyond being an important result on its on, constitutes a solid step into the construction of global  solutions for the  nonlinear model \eqref{modnl} without an appeal to a barrier's method -- as in the previous studies carried in subcritical cases only.
    
    \item[\bf (iv)]  For semilinear  and unstructured equations as \eqref{modnl}, it is routine to consider smallness of initial data for the construction of fixed point solutions. Similar to the global wellposedness result in \cite{bongarti_charoenphon_lasiecka_JEE_2021}, we also relax the smallness assumption on the initial data by requiring it to be small only in $\mathbb{H}$ while having arbitrary $\mathbb{H}_1$--size. This point is very important for further development. 
    
    \item[\bf (v)] Once global nonlinear  solutions are shown  to exist, we seek to stabilize. There are a couple of routes leading to exponential decay of the nonlinear flows. With multi--level energy estimates available for linear solutions, one often seeks to extend such estimates for the nonlinear flow  by taking an  advantage once again of the damping. The main difficulties of this route are \begin{itemize}
    	\item[\bf a)] Neither low nor high energy is dissipative, which will likely require an iterative  "squeezing" argument.
    	
    	\item[\bf b)] Since the relaxed geometrical condition introduces lower order terms, a very careful compactness uniqueness argument would be required to absorb these terms.
    \end{itemize}  

In this paper, however, we exploit another technique which, to the best of our knowledge, is new and makes a strong use of the fact that we \emph{only} require initial data to be small in $\mathbb{H}.$ One of the advantages of such construction (for JMGT) was already exploited by the authors in \cite{bongarti_charoenphon_lasiecka_JEE_2021} in allowing extension by density in nonlinear environments. {\bf \color{blue}In this paper we discovered that it also allows us to \begin{itemize}
		\item[\bf a)] prove exponential stability in a very simple way, basically by the representation of the solution and stability of linear flows. Here the smallness interplay comes to picture through an estimate of the type \begin{equation}\|\cF(\Phi)\|_{\mathbb{H}_1} \leqslant C\|\Phi\|_{\mathbb{H}}\|\Phi\|_{\mathbb{H}_1}.\end{equation}
		
		\item[\bf b)] obtain, to some extent, a continuity property of the decay rate with respect to the $\mathbb{H}$--size of the initial data and the decay rate of the linear flow, say, $\omega_1$. In general, we prove that if $\varepsilon$ is the $\mathbb{H}$--size of the initial data and $\omega(\varepsilon)$ is the corresponding decay rate, then there exists $\underline{\omega}(\varepsilon)$ such that $\omega(\varepsilon) \geqslant \underline\omega(\varepsilon)$ and $\underline\omega(\varepsilon) \to \omega_1$ as $\varepsilon \to 0^+.$ 
	\end{itemize} }
\end{itemize}

{\bf Marcelo-what you are trying to say here?}. Item (i) above is obtained in \cite{bongarti_lasiecka_rodrigues_DCDS_2021}. Since the statement on regularity and representation of solutions is standard once generation of strongly continuous semigroup is established, we  will only prove  (in the next section) generation of semigroup in both $\mathbb{H}$ and $\mathbb{H}_1.$

	Next, we study stability properties of both $S(t)$ and $T(t)$ assuming the general degenerated case for $\gamma$, i.e., $\gamma \in L^\infty(\Omega)$ and $\gamma(x) \geqslant 0$ a.e. in $\Omega.$ This includes the completely degenerate (critical) case when $\gamma =0 $ and the uncontrolled dynamics is unstable.  With  a feedback boundary control, we will be able to show that the semigroups can be stabilized -- but under additional  geometric conditions  which however  are stronger than  the ones typically assumed in  a boundary  stabilization theory of hyperbolic dynamics. 
 It is clear (\Mar{cite Las-Trig}) that  if $\Gamma_0 = \emptyset$ then the entire boundary $\Gamma$ is dissipated and therefore stability results would hold true without  any additional geometric restrictions. 
 Assuming $\Gamma_0 \neq \emptyset,$ and $\Gamma_0$ is star-shaped (standard condition),  classical stability methods (multipliers) do not work due to conflicting 
 signs of  radial vector fields on the boundary $\Gamma_1$.   This fact has been recognized a long time ago \Mar{What is the reference here? cite\{ltz\} Lasiecka-Troggiani-Zhang
 Lasiecka, I.; Triggiani, R.; Zhang, X. Nonconservative wave equations with unobserved Neumann B.C.: global uniqueness and observability in one shot. Differential geometric methods in the control of partial differential equations (Boulder, CO, 1999), 227–325, Contemp. Math., 268, Amer. Math. Soc., Providence, RI, 2000. }.  To remedy the problem we are going to use special constructs to work with  special  geometric constructions which allow to  resolve the dychotomy on the  uncontrolled boundary.
 This, however , requires an additional geometric restriction imposed by {\it convexity}  of uncontrolled part of the boundary $\Gamma_0 $.  Under  convexity and star shaped  requirement  one can  obtain the following:

There exists a vector field $h(x) = [h_1(x), \cdots, h_d(x)] \in C^2 (\overline{\Omega})$ such that 
\begin{subequations}\label{assgeo}
		\begin{equation}\label{assc1}
		    h \cdot \nu = 0 \ \mbox{on} \ \Gamma_0 \end{equation}
	with $\nu$ being the  unit outward normal, and that for some constant $\rho > 0$ and all vector $u(x) \in [L^2(\Omega)]^n$, we have \begin{equation}
				\int_\Omega J(h)|u(x)|^2 d\Omega \geqslant \rho \int_\Omega |u(x)|^2d\Omega,  \label{E77-8}
			\end{equation} where $J(h)$ represents the Jacobian matrix of $h$.
\end{subequations}\begin{remark}
We note  that the more general typical star shaped  condition $h \cdot \nu \leqslant 0$ on $\Gamma_0$  is not sufficient. This is due to the presence of tangential  derivatives on uncontrolled part of the boundary which can not be "absorbed"  via dissipation by  the microlocal argument \Mar{reference-our paper onstabilization of Neumann without geometric conditions
Lasiecka, I.; Triggiani, R. Uniform stabilization of the wave equation with Dirichlet or Neumann feedback control without geometrical conditions. Appl. Math. Optim. 25 (1992), no. 2, 189–224.. Also paper with Tataru.  } which leads to the following local estimate $$|\partial_{\tau} u|_{\Sigma_0} \leq C |u_t|_{\Sigma_0 } +C |\partial_{\nu} u|_{\Sigma_0} +  lot_{Q} $$ valid on solutions. Above, $lot_Q$ mean lower order terms on $Q=\Omega \times [0,T].$ By "bending" on the boundary $\Gamma_0$  a radial vector field allows to eliminate contribution of tangential derivatives on the boundary $\Gamma_0$. See Remark \ref{rmkgeo}. 
\end{remark}
\fi
\ifdefined\xxxxxxxxxxxx
The final result for linear stability is summarized in the following theorem.
\begin{theorem}{\bf [Uniform stability]} \label{thm34n}
 Let $\gamma (x) \geq 0$.  Assume that  $\Gamma_0$ is star-shaped and convex. Then \begin{itemize}
 	\item[\bf (i)] The semigroup $\{S(t)\}_{t \geqslant 0}$ generated by $\cA$ in $\mathbb{H}$ is uniformly exponentially stable with decay rate $\omega_0 > 0.$
 	
 	\item[\bf (ii)] The semigroup  $\{T(t)\}_{t \geqslant 0}$ generated by $\cA$ in $\mathbb{H}_1$ is uniformly exponentially stable with decay rate $\omega_1>0$, where $\omega_1 < \omega_0$.
 \end{itemize} 
\end{theorem}

Once linear wellposedness and uniform stability  of linearization are established, our goal is to prove generation of nonlinear semigroup  on $\mathbb{H}_1 $ by the means of a fixed point argument. To accomplish this, initial data need to be assumed sufficiently small. How small? This is  an important question as argued
  in \cite{bongarti_charoenphon_lasiecka_JEE_2021}. We will be able to show that smallness will be only   imposed at the lowest level of regularity, while higher derivatives can be large. 
  As a consequence,  in the following theorem we show existence of $\mathbb{H}_1$--valued solutions given $\mathbb{H}_1$ initial data which are small in $\mathbb{H}$ only. The proof, given in Section \ref{nonlinearsem}, relies on estimates derived via interpolation inequalities which allows to demonstrate invariance  of a "small" ball under the dynamics. 

We start specifying the notion of solution for the semilinear problem \eqref{modnl} supplemented with \eqref{BC1} and \eqref{IC1}. We denote the initial data here by $\Phi_0 = (u_0,u_1,u_2)^\top.$ Given $T>0$, we say that $$\Phi(t) = (u(t),u_t(t),u_{tt}(t))$$ is a {\bf mild solution} for the system  \eqref{modnl}, \eqref{BC1} and \eqref{IC1} provided $\Phi  \in C([0,T],\mathbb{H}_1)$ and \begin{equation}
\label{formnon}\Phi(t) = T(t)\Phi_0 + \int_0^t T(t-\tau)\cF(\Phi)(\tau)d\tau,
\end{equation}

Before stating the theorem, we denote by $\mathbb{H}^\rho$ (for $\rho > 0$) the set $$\mathbb{H}^\rho := \left\{\Phi \in \mathbb{H}_1; \|\Phi\|_{\mathbb{H}} <\rho \right\}.$$\begin{theorem}[\bf Global Solutions]\label{nsem} Under theb Assumption...... There exists $\rho > 0$ sufficiently small such that, given any $\Phi_0 \in \mathbb{H}^\rho$ the formula \eqref{formnon} defines a continuous $\mathbb{H}_1$--valued mild solution for the system \eqref{modnl}, \eqref{BC1} and \eqref{IC1}. Moreover, for such $\rho > 0$, there exists $R = R(\|\Phi_0\|_{\mathbb{H}_1})$ such that all trajectories starting in $B_{\mathbb{H}^\rho}(0,R)$\footnote{The $\mathbb{H}^\rho$--ball centered at the origin and with radius $R.$} remain in $B_{\mathbb{H}^\rho}(0,R)$ for all $t \geqslant 0.$\end{theorem}

The proof of Theorem \ref{nsem} is accomplished by constructing  a suitable fixed point -- see   Section \ref{nonlinearsem}. It is interesting to notice that the effort to prove exponential stability for both $\mathbb{H}$ and $\mathbb{H}_1$--valued solution now pays off by allowing us to obtain existence of local and global solutions in one shot, without using  Barrier's Method and the associated contradiction argument. Our argument is constructive. 

Once global solutions are shown to exist, we take on the issue of asymptotic (in time) stability. 
The final result is positive, as we expected, and holds uniformly (w.r.t $\gamma$) as long as $\gamma \in L^\infty(\Omega)$ and $\gamma(x) \geqslant 0$ a.e. in $\Omega$. \begin{theorem}[\bf Nonlinear Uniform Stability]\label{expnonl}
	STATE the ASSUMPTIONs. Assume $\gamma \in L^\infty(\Omega)$ and $\gamma(x) \geqslant 0.$ Then there exist $\rho>0$ sufficiently small and $M, \omega > 0$ such that if $\Phi_0 \in \mathbb{H}^\rho$ then \begin{equation}
	\label{expdec} \|\Phi(t)\|_{\mathbb{H}_1} \leqslant Me^{-\omega t}\|\Phi_0\|_{\mathbb{H}_1}, \qquad t \geqslant 0
	\end{equation} where $\Phi$ is the mild solution given by Theorem \ref{nsem}.
\end{theorem} 

\begin{corollary}\label{corexp} With reference to Section 5, let $\beta_0$ be the biggest number such that the map $\Theta$ has a fixed point in $X^{\beta_0}$ which is, moreover, uniformly exponentially stable as in Theorem \ref{expnonl}. Let $\omega: (0,\beta_0] \to \BR_+$ be the function that maps each $\beta>0$ to the decay rate $\omega(\beta)$. Then there exists another function $\underline \omega: (0,\beta_0] \to \BR_+$ such that $\omega(\beta) \geqslant \underline\omega(\beta)$ for all feasible $\beta$ and \begin{equation}
	\label{limexp} \lim\limits_{\beta \to 0} \underline\omega(\beta) = \omega_1,
	\end{equation} where $\omega_1$ is the decay rate of the linear semigroup $T(t).$
\end{corollary}We notice that repeating the statement \textit{there exists $\rho >0$ such that if $\Phi_0 \in \mathbb{H}^\rho$} in the Theorem above is not redundant. In fact, one might need to require even smaller initial data to yield exponential decay. This again highlights the advantage of requiring smallness only in $\mathbb{H}$. For more details, see Section \ref{refext2}.
\fi

The rest of this paper is devoted to the proofs.

\section{Linear Semigroups -- Proof of Theorem \ref {gen}}

We notice that $\mathbb{H}$ has its topology induced by the inner product \begin{equation}
\label{inH} \left(\begin{bmatrix}\xi_1 \\ \xi_2 \\ \xi_3 \end{bmatrix},\begin{bmatrix}\varphi_1 \\ \varphi_2 \\ \varphi_3\end{bmatrix}\right)_{\mathbb{H}} = (\p \xi_1,\p \varphi_1) + \dfrac{b}{\tau}(\p\xi_2,\p \varphi_2) +
(\xi_3,\varphi_3),
\end{equation} for all $(\xi_1,\xi_2,\xi_3)^\top, (\varphi_1,\varphi_2,\varphi_3)^\top \in \mathbb{H}.$

We first show that $\cA: \D(\cA) \subset \mathbb{H} \to \mathbb{H}$ generates a $C_0$ semigroup on $\mathbb{H}.$ This part of the argument follows essentially \cite{bongarti_lasiecka_triggiani_AA_2021} with some  rather straigthforward modifications. We shall outline the main details -- as these are needed  for the proof of generation on the higher level of $\mathbb{H}_1 $ topology. 

 For notational convenience and future use, we introduce the following  change of variables $bz = bu_t + c^2 u$  which reduces the problem to a PDE--abstract ODE coupled system. The change from the coordinates $(u,u_t,u_{tt})$ to $(u,z,z_t)$ is described through the isomophism $M \in \calL(\mathbb{H})$
 given by (see \cite{marchand_triggiani_MMA_2012}) $$M = \begin{bmatrix}1 & 0 & 0 \\ \dfrac{c^2}{b} & 1 & 0 \\ 0 & \dfrac{c^2}{b} & 1 \end{bmatrix}.$$

The next lemma makes the above topological statement precise. \begin{lemma} \label{equiv_prob}  Assume that the compatibility conditions \begin{equation}\label{comp} \lambda{\partial_\nu} u_0 + \kappa_0 u_0 = 0 \ \emph{\mbox{on}} \ \Gamma_0, \qquad {\partial_\nu} u_0 + \kappa_1 u_1 = 0 \ \emph{\mbox{on}} \ \Gamma_1
	\end{equation} hold.
	Then $\Phi \in C^1(0,T;\mathbb{H})\cap C(0,T;\D(\cA))$ is a strong solution of\begin{equation} \begin{cases}
	\label{usistphi} \Phi_t = \mathcal{A}\Phi\\ 
	\Phi(0) = \Phi_0 \end{cases}
	\end{equation} if, and only if, $\Psi= M\Phi \in C^1(0,T;\mathbb{H})\cap C(0,T;\D(\mathbb{A}))$ is a strong solution for  \begin{equation} \begin{cases}
	\label{usistz} \Psi_t = \mathbb{A}\Psi\\ 
	\Psi(0) = \Psi_0 = M\Phi_0\end{cases}
	\end{equation} where $\mathbb{A} = M\cA M^{-1}$ with \begin{equation}\label{domz}
	\D(\mathbb{A}) = \left\{\begin{bmatrix} \xi_1 \\ \xi_2 \\ \xi_3 \end{bmatrix} \in \left[H^2(\Omega)\right]^2 \times \D \left(\p\ \right);
	\left[\lambda \partial_\nu \xi_2 + \kappa_0\xi_2\right]_{\Gamma_0} =0, \left[\partial_\nu \xi_2 + \kappa_1\xi_3\right]_{\Gamma_1} =0
	\right\}
	\end{equation}
\end{lemma} \begin{proof}
We only check the matching of the boundary conditions. Assume that $\Psi = (u,z,z_t) \in C^1(0,T;\mathbb{H}) \cap C(0,T;\D(\mathbb{A}))$ is a strong solution for \eqref{usistz}. Let 
$$\Upsilon(t) := \left(\lambda \partial_\nu u(t) + \kappa_0 u(t)\right)\rvert_{\Gamma_0}, \ t \geqslant 0$$ and notice that $b\Upsilon_t +c^2\Upsilon = 0$ for all $t$. This along with the compatibility condition \eqref{comp}$_1$ ($\Upsilon(0) = 0$) implies that $\Upsilon \equiv 0$. The same argument \emph{mutatis mutandis} recovers the boundary condition for $u$ on $\Gamma_1.$
\end{proof}

\

For convenience, we explicitly write a formula for the new operator $\mathbb{A} = M\cA M^{-1}.$ We have 
\begin{align}
\mathbb{A}\begin{bmatrix} \xi_1 \\[2mm] \xi_2 \\[2mm] \xi_3 \end{bmatrix} = 	\begingroup 
\setlength\arraycolsep{12pt} \begin{bmatrix} -\dfrac{c^2}{b}I & I & 0 \\[2mm] 0 & 0 & I \\[2mm] -\gamma \dfrac{c^4}{\tau b^2}I & \gamma \dfrac{c^2}{\tau b}I - \dfrac{b}{\tau}A & -\gamma \dfrac{1}{\tau}I - \dfrac{b}{\tau}AN(\kappa_1 N^\ast A)\end{bmatrix} 	\endgroup \begin{bmatrix} \xi_1 \\[2mm] \xi_2 \\[2mm] \xi_3 \end{bmatrix}
\end{align} 
where $\gamma = \alpha - \dfrac{\tau c^2}{b} \in L^\infty(\Omega).$

\medskip

We are ready to prove Theorem \ref{gen}. This  will be done by first showing that $\mathbb{A}$ generates a $C_0$--semigroup on $\mathbb{H}$, from which the semigroup generated by $\cA$ can be recovered via $M$. The semigroup on $\mathbb{H}_1$ will then be obtained by a restriction argument. The details are below. 

We write $\mathbb{A} = \mathbb{A}_d + P$ where $$P := \begin{bmatrix}0 & 1 & 0 \\ 0 & 0 & 0 \\ -\gamma \dfrac{c^4}{\tau b^2} & \gamma\dfrac{c^2}{\tau b} & 1-\gamma\dfrac{1}{\tau}\end{bmatrix} \in \calL(\mathbb{H})$$ is bounded in $\mathbb{H}$ and \begin{equation}\mathbb{A}_d\begin{bmatrix} \xi_1 \\[2mm] \xi_2 \\[2mm] \xi_3 \end{bmatrix} :=\begingroup 
	\setlength\arraycolsep{12pt}  \begin{bmatrix} -\dfrac{c^2}{b}I & 0 & 0 \\[2mm] 0 & 0 & I \\[2mm] 0 & - \dfrac{b}{\tau}A & -I - \dfrac{b}{\tau}AN(\kappa_1N^\ast A)\end{bmatrix}\endgroup \begin{bmatrix} \xi_1 \\[2mm] \xi_2 \\[2mm] \xi_3 \end{bmatrix},\end{equation} where $\D(\mathbb{A}_d) := \D(\mathbb{A}).$ It then suffices to prove generation of $\mathbb{A}_d$ on $\mathbb{H}$, see \cite[p. 76]{pazy_S_1992} and this will be done  by verifying the hypothesis of Lummer Philips Theorem: dissipativity and maximality.

For {\it dissipativity} we consider $(\xi_1,\xi_2,\xi_3)^\top \in \D(\mathbb{A})$ and compute via \eqref{inH}\begin{align*}
\left(\mathbb{A}_d\begin{bmatrix} \xi_1 \\[2mm] \xi_2 \\[2mm] \xi_3 \end{bmatrix},\begin{bmatrix} \xi_1 \\[2mm] \xi_2 \\[2mm] \xi_3 \end{bmatrix}\right)_{\mathbb{H}} &= - \dfrac{c^2}{b}\|\p \xi_1\|_{L^2(\Omega)}^2 +  \dfrac{b}{\tau}\left(\p \xi_3, \p\xi_2\right)  \\ & -\|\xi_3\|_{L^2(\Omega)}^2 - \dfrac{b}{\tau}(\p\xi_2,\p\xi_3)  -\dfrac{b}{\tau}\|\sqrt{\kappa_1}\xi_3\|_{L^2(\Gamma_1)}^2 \\ &= - \dfrac{c^2}{b}\|\p \xi_1\|_{L^2(\Omega)}^2 -\|\xi_3\|_{L^2(\Omega)}^2 -\dfrac{b}{\tau}\|\sqrt{\kappa_1}\xi_3\|_{L^2(\Gamma_1)}^2 \leqslant 0,
\end{align*} hence, $\mathbb{A}_d$ is dissipative in $\mathbb{H}$.

For {\it maximality} in $\mathbb{H}$, given any $L = (f,g,h) \in \mathbb{H}$ we need to find $\Psi = (\xi_1,\xi_2,\xi_3)^\top \in \D(\mathbb{A})$ such that $(s - \mathbb{A}_d)\Psi = L$, for some $s >0.$ This is equivalent to solving \begin{equation}\label{systm-max}
\begin{cases}
s \xi_1 + \dfrac{c^2}{b}\xi_1 = f, \\ s \xi_2 - \xi_3 = g, \\
s \xi_3 +\xi_3 + b A(\xi_2+ N(\kappa_1N^*A\xi_3))= h,
\end{cases}
\end{equation} which readily implies $$\xi_1 = \dfrac{b}{bs + c^2}f \in \D(\p).$$ Moreover, since $A^{-1} \in \mathcal{L}({L^2(\Omega)})$ a combination of the second and third equations above yields \begin{equation}
    \label{sol2} K_s \xi_3 = s A^{-1}h - bg
\end{equation} where $K_s: L^2(\Omega) \to L^2(\Omega)$ acts on an element $\xi \in L^2(\Omega)$ as $$K_s \xi = \left[(s^2 + s)A^{-1}+b(I+s N(\kappa_1N^\ast A)\right]\xi.$$ 

We now notice the restriction $K_s\rvert_{\D(\p)}$ is strictly positive. Indeed it follows by \eqref{neq} that, given $\xi \in \D(\p)$ we have 
\begin{align*}
    (K_s \xi, \xi)_{\D(\p)} &= (s^2 + s)\|\xi\|_2^2 + b\|\p\xi\|_2^2 + bs(A^{1/2}N(\kappa_1 N^\ast A)\xi,A^{1/2}\xi) \\ &= (s^2 + s)\|\xi\|_2^2 + b\|\p\xi\|_2^2 + bs\|\sqrt{\kappa_1}N^\ast A\xi\|_{\Gamma_1}^2 > 0.
\end{align*} 

Therefore
${K_s}_{\D(\p)}^{-1} \in \mathcal{L}(\D(\p))$ and since $s A^{-1}h - bg \in \D(\p)$ we have that $$\xi_3 := K_s^{-1}(s A^{-1}h - bg) \in \D(\p)$$ is the solution of \eqref{sol2}. Finally, $$\xi_2 = s^{-1}(\xi_3+g)  
\in \D(\p).$$

For the final step to conclude membership of $(\xi_1,\xi_2,\xi_3)$ in $\D(\mathbb{A})$ we look at the \textit{abstract} version of the description of $\D(\mathbb{A})$: \begin{equation*}\
	\D(\mathbb{A}) = \left\{(\xi_1,\xi_2,\xi_3)^\top \in \mathbb{H}; \ \xi_2+N(\kappa_1N^*A\xi_3) \in \D(A)
	\right\}
	\end{equation*} whereby one only needs to check that $\xi_2+N(\kappa_1N^*A\xi_3) \in \D(A)$ since the regularity for the triple $(\xi_1,\xi_2,\xi_3)^\top$ to belong to $\mathbb{H}$ was already established. The desired regularity will follow from \eqref{systm-max}, which implies \begin{align*}
bs (\xi_2+N(\kappa_1N^*A\xi_3)) &= - (s^2 + s)A^{-1}\xi_3 + s A^{-1} h \in \D(A),
\end{align*} since $\xi_3, h \in L^2(\Omega).$ 

This proves that $\mathbb{A}_d$ is maximal dissipative, therefore generates a $C_0$--semigroup of contractions due to Lummer Phillips Theorem and, since $P$ is bounded, $\mathbb{A} = \mathbb{A}_d + P$ generates a $C_0$--semigroup on $\mathbb{H}.$  

For generation in $\mathbb{H}_1$ we use an argument inspired by the one presented in (\cite{marchand_triggiani_MMA_2012}, p. 26) with the needed modifications. Since we already know that $\A$ generates a $C_0$ semigroup $\{S(t)\}_{t \geqslant 0}$ on a larger space $\mathbb{H}$, we only show that $$\{T(t)\}_{t \geqslant 0} := \{S(t)\rvert_{\mathbb{H}_1}\}_{t \geqslant 0}$$ is also a semigroup and that its infinitesimal generator is $\A$ when considered as an operator in $\mathbb{H}_1.$ 

This entails the proof of two things: $\{T(t)\}_{t \geqslant 0}$ satisfies the semigroup property -- which follows from the fact that the problem is autonomous -- and invariance: $T(t)(\mathbb{H}_1) \subset \mathbb{H}_1$ for all $t \geqslant 0$. 

If $\Phi_0=(u_0,u_1,u_2)^\top \in \mathbb{H}_1 $ then $\partial_\nu u_0 + k_1 u_1 =0 $ on $\Gamma_1$. We then need to show that this condition is invariant under the dynamics and, in addition, the regularity  $\Delta u \in C\left([0,T);L^2(\Omega)\right)$ holds true. This along with the boundary conditions and regularity of elliptic problems will prove that  $u\in C([0,T); H^2(\Omega))$.   In order to show that $\Delta u \in C\left([0,T);L^2(\Omega)\right)$, we appeal to the change of variables $bz_t = bu_t + c^2 u$. By the variation of parameters formula we have $$u(t) = e^{-\frac{c^2}{b} t} u_0 + \int_0^t e^{-\frac{c^2}{b}(t-\sigma)}z(\sigma)d\sigma$$ and since $\Delta u_0 \in L^2(\Omega)$ ($\Phi_0 \in \mathbb{H}_1$), it suffices to verify that $$\int_0^t e^{-\frac{c^2}{b}(t-\sigma)} \Delta z(\sigma)d\sigma \in L^2(\Omega), \ \forall t \geqslant 0.$$ To this end we  recall that $ (z,z_t) \in C([0,T); H^1(\Omega)\times L^2(\Omega)).$  Writing the linear  solution of  \eqref{modnl} (with $k = 0$) in the $z$--variable yields \begin{align}\label{ode}\int_0^t e^{-\frac{c^2}{b}(t-\sigma)} \Delta z(\sigma)d\sigma &= \dfrac{\tau}{b}\int_0^t e^{-\frac{c^2}{b}(t+\sigma)}\left[z_{tt}(\sigma)+\gamma u_{tt}(\sigma)\right]d\sigma \nonumber \\ &= \dfrac{\tau}{b} \left[z_t(t)+\gamma u_t(t) - e^{-\frac{c^2}{b}t}[z_t(0)+\gamma u_1] \right] \nonumber \\ &+ \dfrac{c^2}{b^2}\int_0^t e^{-\frac{c^2}{b}(t-\sigma)}[z_t(\sigma)+\gamma u_t(\sigma)]d\sigma \in L^2(\Omega),\end{align} as needed. Now, in order to show that the boundary conditions are also invariant under the dynamics, we write (for continuous $\calD(\cA)$--valued solutions):  \begin{align*}
	\partial_{\nu} u(t) + k_1 u_t(t) &= e^{-\frac{c^2}{b}t} \partial_{\nu} u_0 +\int_0^t  e^{-c^2b^{-1}( t-\sigma) } \partial_{\nu}  z(\sigma) d\sigma \\ &+\kappa_1 \left(- \dfrac{c^2} {b}  e^{-\frac{c^2}{b} t} u_0 + 
	e^{-\frac{c^2}{b}t} z(0) + \int_0^t  e^{-\frac{c^2}{b}( t-\sigma) }  z_t(\sigma) d\sigma \right) \\ &= e^{-\frac{c^2}{b}t} \partial_{\nu} u_0 +\int_0^t  e^{-c^2b^{-1}( t-\sigma) } \partial_{\nu}  z(\sigma) d\sigma \\ &+k_1 \left[- \dfrac{c^2} {b}  e^{-\frac{c^2}{b} t} u_0 + 
	e^{-\frac{c^2}{b}t} \left(u_1 + \dfrac{c^2}{b}u_0\right) + \int_0^t  e^{-\frac{c^2}{b}( t-\sigma) }  z_t(\sigma) d\sigma \right] \\ &= e^{-\frac{c^2}{b}t}\left[ \partial_{\nu} u_0 + \kappa_1u_1\right] +\int_0^t  e^{-c^2b^{-1}( t-\sigma) } \left[\partial_{\nu}  z(\sigma) + \kappa_1 z_t(\sigma)\right] d\sigma = 0,
\end{align*} where the conclusion follows from the fact that the initial conditions for $u$ satisfy the absorbing boundary conditions and the variable $z$ satisfies the absorbing boundary conditions along the trajectory. 
 This completes the proof of Theorem \ref{gen}.
	
\section{Exponential decays -- Proof of Theorem \ref{thm34n}} \label{secsta} In this section we work with the linearized version of \eqref{modnl} -- i.e., we take $k = 0$ -- in the $z$--variable. Moreover, since $\tau$ is fixed, we lose no generality by setting $\tau = 1$, therefore this is assumed for the rest of the paper. Recall that the change of variables $z = u_t + \frac{c^2}{b}u$ transforms \eqref{modnl2} into
\begin{equation}\label{zprob}
    z_{tt} + bA(z_t + N(\kappa_1N^*Az)) = -\gamma u_{tt} + f.
\end{equation} We assume smooth initial conditions and a $u$--independent forcing term $f \in C^1(\BR_+, L^2(\Omega))$. This ensures the existence and uniqueness of classical solutions which are, in addition, continuously dependent on the initial data. We can eventually extend the results that follow to semigroup solutions by the virtue of density. 

\

The energy for a solution $\Phi = (u,u_t,u_{tt}) \in \calD(\A)$ will be computed in two levels. We define the {\bf lower energy} functional $E(t) = E_0(t) + E_1(t)$ where \begin{align}\label{E1} E_1(t) & :=
	\dfrac{b}{2}\|\p{z}\|_2^2 +  \dfrac{1}{2}\|z_t\|_2^2 + \dfrac{c^2}{2b}\|\gamma^{1/2}u_t\|_2^2\end{align} \begin{align}\label{E0} E_0(t)& := \dfrac{1}{2}\|\alpha^{1/2}u_t\|_2^2 + \dfrac{c^2}{2}\|\p{u}\|_2^2\end{align} and the {\bf higher energy} functional $\E(t) = E(t) + E_2(t)$ where \begin{equation}\label{E2}E_2(t) = \dfrac{b}{2}\|\Delta u\|_2^2 
	\end{equation} where $\alpha$ and $\gamma$ may depend on $x.$
	We also note that $$\|A^{1/2} u\|^2_2 = \|\nabla u\|^2_2 + \|\sqrt{\kappa_0}u\|^2_{\Gamma_0} $$.
	By Poincare- Wirtenberg inequality we obtain that $$\|A^{1/2} u\|_2 \sim \|u\|_{H^1(\Omega) } $$
	
It is standard to see that $E(t) \sim \|\Phi(t)\|_{\mathbb{H}}^2$, see \cite{bongarti_lasiecka_DSOCIP_2021} for  details. The following lemma follows from classic elliptic theory and provides the estimate which is necessary for justifying our choice of \textit{higher} energy functional. As a remark, here and hereafter we use the notation $a \lesssim b$ to say that $a \leqslant Cb$ where $C$ is a constant possibly depending on the physical parameters of the model ($\tau,c,b > 0$) but independent of space, time and $\gamma \in L^\infty(\Omega).$\begin{lemma} Let $\Omega$ be a smooth domain and consider a function $u: \Omega \to \mathbb{R}$ such that $\Delta u \in L^2(\Omega)$ and $\partial_\nu u \rvert_{\partial \Omega} \in H^{1/2}(\Gamma)$, then $u \in H^2(\Omega)$ and $$\|u\|_{H^2(\Omega)} \lesssim \|\Delta u\|_2+ \|\partial_\nu u\|_{H^{1/2}(\Gamma)}.$$ \end{lemma}
	\ifdefined\xxxxxxxx
	{\bf Marcelo-what is this nonsens? I am not sure what is in your head. You do not have Dirichlet data.} In particular, if $(u,u_t,u_{ttt})$ is a classical solution, then $u \in H_{\Gamma_1}^2(\Omega)$ and the inequality above along with the trace theorem\footnote{Trace Theorem: $\|f\|_{\Gamma}^2 \lesssim \|\p{f}\|_2^2.$} give \begin{equation}\label{E2L}\|u\|_{H_{\Gamma_1}^2(\Omega)}^2 \lesssim \|\Delta u\|_2^2 + \|u_t\|_{H_{\Gamma_1}^1(\Omega)}^2 \sim \|\Delta u\|_2^2 + \left\|\p u_t\right\|_{2}^2.\end{equation}
\end{lemma}
This  established topological  equivalence between $\E(t)$ and $\|\Phi(t)\|_{\mathbb{H}_1}^2$.

\begin{remark}
	Notice that the term $\|\p u_t\|_2^2$ in \eqref{E2L}  is not needed  in \eqref{E2} simply because $$\|\p u_t\|_2^2 = \left\|\p z - \dfrac{c^2}{b}\p u\right\| \lesssim E(t).$$ 
\end{remark}
\fi
\ifdefined\xxxx The next lemma guarantees that stability of solutions in $\mathbb{H}$ is equivalent to uniform exponential decay of the function $t \mapsto E(t).$

\begin{lemma}\label{equinorm}
	Let $\Phi = (u,u_t,u_{tt})$ be a weak solution for the $u$-- problem in $\mathbb{H}$ and assume that \eqref{comp} is in force. Then the following statements are equivalent: \begin{itemize}
		\item[\bf a)] $t \mapsto \|\Phi(t)\|_{\mathbb{H}}^2$ decays exponentially.
		
		\item[\bf b)] $t \mapsto \|M\Phi(t)\|_{\mathbb{H}}^2 = \|(u,z,z_t)\|_{\mathbb{H}}^2$ decays exponentially.
		
		\item[\bf c)] $t \mapsto E(t)$ decays exponentially.
	\end{itemize}
\end{lemma} \begin{proof} See \cite{bongarti_lasiecka_DSOCIP_2021}.\end{proof}

\begin{remark}
The purpose of Lemma \eqref{equinorm} is that it allows us to use both the expression of the energy $E(t) = E_0(t) + E_1(t)$ and the norm of the solution $\|(u,z,z_t)\|_{\mathbb{H}}^2$ interchangeably. The specific structure of the energy contributes to a discovery of certain invariances and dissipative laws. 
However, from the topological point of view it is essential that 
 the following three quantities 
  $\|\nabla z\|_2,$
  $\|z_t\|_2$%
  \ and 
  $\|\nabla u\|_2$
  display appropriate decays. 
\end{remark
}\fi

\subsection{ Propagation of boundary dissipation -- Flow multipliers }
We begin with energy identity for $E_1.$  Since the problem is linear we work with smooth solutions (in the domain of the generator) which can be extended by density to  the phase space solutions. \begin{proposition}[\bf Energy Identity]
    \label{ene1id} Let $T>0$. If $\Psi = (u,z,z_t)$ is a weak  solution of \eqref{zprob} then 
	    \begin{equation}\label{e1id}
		    E_1(T) + \int_t^T D_\Psi(s)ds
		    = E_1(t) + \int_t^T\int_\Omega fz_t d\Omega ds
	    \end{equation} holds for $0\leq t\leq T$, where $D_\Psi$ represents the interior/boundary damping and is given by \begin{equation}
	        \label{damp} D_\Psi := b\int_{\Gamma_1} \kappa_1z_t^2d\Gamma_1
		    + \int_\Omega \gamma u_{tt}^2 d\Omega ds
	    \end{equation}
\end{proposition} 
 \begin{proof}The energy identity \eqref{e1id} is first derived  for strong solutions and then extended  by density  to weak solutions. 
         Consider the bilinear form $\langle\cdot,\cdot\rangle: \mathbb{H} \times \mathbb{H} \to \mathbb{R}$ be given by
\begin{equation}\begin{aligned}\label{biform}
    \left\langle \begin{bmatrix} \xi_1 \\ \xi_2 \\ \xi_3 \end{bmatrix}, \begin{bmatrix} \varphi_1 \\ \varphi_2 \\ \varphi_3\end{bmatrix}\right\rangle := 
    b\left(A^{1/2}\xi_2,A^{1/2}\varphi_2\right)
    + \left(\xi_3,\varphi_3\right)
    + \frac{c^2}{b}\left(\gamma\left(\xi_2-\frac{c^2}{b}\xi_1\right),\varphi_2-\frac{c^2}{b}\varphi_1\right)
\end{aligned}\end{equation}
which is continuous. Moreover, recalling that $\Psi=(u,z,z_t)$ it follows that $2E_1(t) = \langle\Psi(t),\Psi(t)\rangle$. Therefore, with  $G = (0,0,f)^\top$ we have, after   straightforward  calculations  (see \cite{bongarti_lasiecka_rodrigues_DCDS_2021} for details)
\begin{align*}
    2\dfrac{dE_1(t)}{dt} 
    &= \left\langle\dfrac{d\Psi(t)}{dt},\Psi(t)\right\rangle
    = \left\langle\mathbb{A}\Psi(t) + G,\Psi(t)\right\rangle \\
   &= -\int_\Omega \gamma u_{tt}^2d\Omega - b \int_{\Gamma_1}\kappa_1z_t^2d\Gamma_1 + \int_\Omega fz_td\Omega.
\end{align*} Identity \eqref{e1id} then follows by an integration in time on $[t,T].$
\end{proof}

In the next step we reconstruct the integral of the full energy on a truncated time interval  $(s,T-s)$ for $0 < s < T/2.$

\begin{proposition}\label{id}
	Let $T>0$. If $(u,z,z_t)$ is a classical solution of \eqref{usistz} then the inequality 
	    \begin{align}\label{e1in1}
	        \int_s^{T-s} E_1(t) dt 
	        &\lesssim [E_1(s)+E_1(T-s)] + 
	        \nonumber \\ &C_T\left[\int_0^T D_\Psi(s)ds
	        +\int_Q f^2 dQ + lot_\delta(z)\right].\end{align} holds for $0<s<T/2$. Here, $lot_\delta(z)$ is a collection of lower order terms satisfying $$lot_\delta(z)\leq C_\delta\sup_{t\in[0,T]}\left\{\|z(t)\|_{H^{1-\delta}(\Omega)}^2+\|z_t(t)\|_{H^{-\delta}(\Omega)}^2\right\},$$ for some $\delta>0$.
\end{proposition} \begin{proof} We study stability properties of both $S(t)$ and $T(t)$ assuming the general degenerated case for $\gamma$, i.e., $\gamma \in L^\infty(\Omega)$ and $\gamma(x) \geqslant 0$ a.e. in $\Omega.$ This includes a  completely degenerate (critical) case when $\gamma =0 $ and the uncontrolled dynamics is unstable.  With  a feedback boundary control, we will be able to show that the semigroups can be stabilized -- but under additional  geometric conditions  which however  are stronger than  the ones typically assumed in a boundary stabilization theory of hyperbolic dynamics. It is clear that  if $\Gamma_0 = \emptyset$ then the entire boundary $\Gamma$ is dissipated and therefore stability results would hold true without  any additional geometric restrictions. 
 Assuming $\Gamma_0 \neq \emptyset,$ and $\Gamma_0$ is star-shaped (standard condition),  classical stability methods (multipliers) do not work due to conflicting signs of  radial vector fields on the boundary $\Gamma_1$ when acting on tangential derivatives.   This fact has been recognized already in  \cite{lasiecka_triggiani_zhang_CM_2000}.
 However,due to  convexity of $\Gamma_0$  and Assumption \ref{geo}  one obtains the following construction \cite{lasiecka_triggiani_zhang_CM_2000}: there exists a vector field $h(x) = [h_1(x), \cdots, h_d(x)] \in C^2 (\overline{\Omega})$ such that 
\begin{subequations}\label{assgeo}
		\begin{equation}\label{assc1}
		    h \cdot \nu = 0 \ \mbox{on} \ \Gamma_0 \end{equation}
	with $\nu$ being the  unit outward normal, and that for some constant $\delta > 0$ and all vector $u(x) \in [L^2(\Omega)]^n$, we have \begin{equation}
				\int_\Omega J(h)|u(x)|^2 d\Omega \geqslant \delta  \int_\Omega |u(x)|^2d\Omega,  \label{E77-8}
			\end{equation} where $J(h)$ represents the Jacobian matrix of $h$.
\end{subequations}\begin{remark}
We note  that the more general typical star--shaped  condition $h \cdot \nu \leqslant 0$ on $\Gamma_0$  is not sufficient. This is due to the presence of tangential  derivatives on uncontrolled part of the boundary which can not be ``absorbed''  via dissipation by  the microlocal argument \cite{lasiecka_triggiani_AMO_1992,tataru_ASNSPCS_1998,lasiecka_tataru_DIE_1993}. 


This all leads to the following local estimate $$\|\partial_{\tau} u\|_{\Sigma_0} \leq C \|u_t\|_{\Sigma_0 } +C \|\partial_{\nu} u\|_{\Sigma_0} +  lot_{Q} $$ valid on solutions. Above, $lot_Q$ mean lower order terms on $Q=\Omega \times [0,T].$   By ``bending'' on the boundary $\Gamma_0$  the  radial vector field allows to eliminate its contribution of tangential derivatives. See Remark \ref{rmkgeo}. 
\end{remark}
\begin{remark}\label{other} Convexity of $\Gamma_0$ is only one sufficient condition. Several examples where the construction in (\ref{assgeo}) holds for  other types of domains are given in \cite{lasiecka_triggiani_zhang_CM_2000}.

As mentioned earlier, since we have dissipation  only on a portion of the boundary, while the remaining part is subject to Neuman/Robin type of boundary conditions,, standard multiplier theory with radial vector fields  does no apply- due to sign inconsistency in the estimates of tangential  boundary integrals. The introduction of the vector field $h \in C^2$ in \eqref{assgeo} is then  critical for the results in this paper. 
\end{remark}

We start with   energetic calculations performed first on regular (strong) solutions.  Let's multiply equation \eqref{zprob} by $h\cdot\nabla{z}$ integrate by parts in $(s,T-s)\times\Omega$, for $s \in [0,T/2).$ This gives
\begin{align}
	 &\dfrac{1}{2}\int_s^{T-s}\!\!\!\!\int_{\Gamma} \left(z_t^2 - b|\nabla z|^2\right)(h\cdot\nu) d\Gamma dt + b\int_s^{T-s}\!\!\!\!\int_{\Gamma} \partial_\nu z(h\cdot\nabla z) d\Gamma dt \nonumber \\ &= \int_s^{T-s}\!\!\!\!\int_\Omega \gamma u_{tt} (h\cdot\nabla z)d\Omega dt + \int_\Omega z_t(h\cdot\nabla z)d\Omega\biggr\rvert_s^{T-s}  \nonumber \\ &+ \dfrac{b}{2}\int_s^{T-s}\!\!\!\!\int_\Omega J(h)|\nabla z|^2d\Omega dt + \dfrac{1}{2}\int_s^{T-s}\!\!\!\!\int_\Omega \left(z_t^2 - b|\nabla z|^2\right)\dvv{h}d\Omega dt \nonumber \\ 
	& - \int_s^{T-s}\!\!\!\!\int_\Omega f(h\cdot\nabla z)d\Omega dt, \label{hgradz}
\end{align} where $J(h)$ is the Jacobian. Now notice that  equipartition  of kinetic and potential energy appears in  the above identity via the term \begin{equation}\label{aux1}\int_\Omega \left(z_t^2 - b|\nabla z|^2\right)\dvv{h}d\Omega,\end{equation} . We next  multiply \eqref{zprob} by $z\dvv{h}$ and again integrate by parts to obtain the following identity \begin{align}
	&\dfrac{b}{2} \int_s^{T-s}\!\!\!\!\int_\Gamma \partial_\nu z z \mbox{div}(h)d\Gamma dt       = \dfrac{1}{2}\int_s^{T-s}\!\!\!\!\int_\Omega \gamma u_{tt} z \mbox{div}(h) d\Omega dt \nonumber \\ &+ \dfrac{1}{2} \int_\Omega z_tz \mbox{div}(h)d\Omega\biggr\rvert_s^{T-s} + \dfrac{1}{2}\int_s^{T-s}\!\!\!\!\int_\Omega \left(b|\nabla z| - z_t^2\right)\mbox{div}(h)d\Omega dt \nonumber \\ &- \dfrac{b}{2} \int_s^{T-s}\!\!\!\!\int_\Omega z \nabla z \cdot \nabla(\mbox{div}(h))d\Omega dt - \dfrac{1}{2}\int_s^{T-s}\!\!\!\!\int_\Omega fz\dvv{h}d\Omega dt \label{zdivh}
\end{align} where we have, as in \eqref{hgradz}, kept the boundary terms on the left--hand--side.
  Adding \eqref{hgradz} and \eqref{zdivh} we obtain \begin{align*}
	 &\dfrac{1}{2}\int_s^{T-s}\!\!\!\!\int_{\Gamma} \left(z_t^2 - b|\nabla z|^2\right)(h\cdot\nu) d\Gamma dt + b\int_s^{T-s}\!\!\!\!\int_{\Gamma} \partial_\nu z(h\cdot\nabla z) d\Gamma dt + \dfrac{b}{2} \int_s^{T-s}\!\!\!\!\int_\Gamma \partial_\nu z z \mbox{div}(h)d\Gamma dt      \nonumber \\ &= \int_s^{T-s}\!\!\!\!\int_\Omega \gamma u_{tt} (h\cdot\nabla z)d\Omega dt +\dfrac{1}{2}\int_s^{T-s}\!\!\!\!\int_\Omega \gamma u_{tt} z \mbox{div}(h) d\Omega dt+ \dfrac{b}{2}\int_s^{T-s}\!\!\!\!\int_\Omega J(h)|\nabla z|^2d\Omega dt \nonumber \\ & + \int_\Omega z_t(h\cdot\nabla z)d\Omega\biggr\rvert_s^{T-s} + \dfrac{1}{2} \int_\Omega z_tz \mbox{div}(h)d\Omega\biggr\rvert_s^{T-s}  - \dfrac{b}{2} \int_s^{T-s}\!\!\!\!\int_\Omega z \nabla z \cdot \nabla(\mbox{div}(h))d\Omega dt \nonumber \\ &- \int_s^{T-s}\!\!\!\!\int_\Omega f(h\cdot\nabla z)d\Omega dt - \dfrac{1}{2}\int_s^{T-s}\!\!\!\!\int_\Omega fz\dvv{h}d\Omega dt,
\end{align*} and then the boundary terms can be written more compactly in terms of the interior terms as \begin{align}
	 \int_s^{T-s}B(\Gamma)(t)dt &= \int_s^{T-s}\!\!\!\!\int_\Omega \left(\gamma u_{tt} - f\right) M_h(z)d\Omega dt + \int_\Omega z_tM_h(z)d\Omega\biggr\rvert_s^{T-s} \nonumber \\ &+ \dfrac{b}{2}\int_s^{T-s}\!\!\!\!\int_\Omega J(h)|\nabla z|^2d\Omega dt - \dfrac{b}{2} \int_s^{T-s}\!\!\!\!\int_\Omega z \nabla z \cdot \nabla(\mbox{div}(h))d\Omega dt. \label{hmult1}
\end{align} by defining $M_h(z):= h\cdot \nabla z + \frac{1}{2}z\dvv{h}$ and \begin{equation}
    \label{bddt} B(\Gamma) := \dfrac{1}{2}\int_{\Gamma} \left(z_t^2 - b|\nabla z|^2\right)(h\cdot\nu) d\Gamma + b\int_{\Gamma} \partial_\nu zM_h(z) d\Gamma.  
\end{equation}

Now, the second part of  geometrical condition \eqref{assgeo} allow us to obtain an estimate for the potential $z$--energy. To see this, first notice that $M_h(z)$ is controlled by the potential energy, indeed, \begin{equation}\label{mhest}\|M_h(z)|_2| \leqslant \sup\limits_{x \in \overline{\Omega}} (|h(x)|+\dvv{h}(x)|)\left(\|\nabla z\|_2 + \frac{1}{2}\|z\|_2\right) \lesssim \|z\|_{H^1(\Omega) } \lesssim  \|A^{1/2}  z\|_2 ,\end{equation} due to Robin boundary condition imposed on $\Gamma_0$ which allows to control $L_2$ norms  by the gradient.  Moreover, the last term in \eqref{hmult1} can be estimated as \begin{equation}\label{est1-2}
    \left|\int_s^{T-s}\!\!\!\!\int_{\Omega} z \nabla z \nabla(\mbox{div}(h)) d\Omega dt \right|
    \lesssim \varepsilon \int_s^{T-s}\!\!\!\!\int_{\Omega}|\nabla{z}|^2 d\Omega dt + C_\varepsilon \|z\|_{L^2(s,T-s;L^2(\Omega))}^2,
\end{equation} for any $\varepsilon > 0$, due to Peter--Paul's Inequality and boundedness of $D^2h$ in $\overline \Omega.$ Similarly, for any $\varepsilon > 0$ \begin{equation*}
    \left|\int_s^{T-s}\!\!\!\!\int_{\Omega} \gamma u_{tt} A_h(z) d\Omega dt\right|
    \lesssim \varepsilon\int_s^{T-s}\!\!\!\!\int_{\Omega}|\nabla{z}|^2 d\Omega dt
    + C_\varepsilon\int_s^{T-s}\!\!\!\!\int_{\Omega}\gamma|u_{tt}|^2d \Omega dt
\end{equation*} 
and finally \begin{align}
    \label{aux2} \int_\Omega z_tM_h(z)d\Omega\biggr\rvert_s^{T-s} &\lesssim E_1(s) + E_1(T-s). 
\end{align} Therefore, combining \eqref{mhest}, \eqref{est1-2}, \eqref{aux2} and the second part of assumption \eqref{assgeo} we obtain, for $\varepsilon > 0$ sufficiently small, \begin{align}\label{potz} \int_s^{T-s}\!\!\!\!\int_{\Omega}|\nabla{z}|^2 d\Omega dt
    &\lesssim E_1(s)+E_1(T-s) 
    + \int_0^T \!\!\!\!D_\Psi(s)ds \nonumber \\ &+ \int_s^{T-s}\!\!\!\!B(\Gamma)(t)dt + \int_Qf^2dQ
    + \|z\|_{L^2(s,T-s;L^2(\Omega))}^2.
\end{align}

Now, given the definition of $E_1$, we need an estimate for the kinetic energy in order to continue. For that, we multiply \eqref{zprob} by $z$ and again integrate by parts over $(s,T-s)\times\Omega$ to obtain 
\begin{align}
	&\int_s^{T-s}\!\!\!\!\int_{\Omega} \left[b|\nabla{z}|^2 - z_t^2\right]d\Omega dt + b\int_s^{T-s}\!\!\!\!\int_{\Gamma} z\partial_\nu d\Gamma dt \nonumber = \\ 
	& - \int_s^{T-s}\!\!\!\!\int_{\Omega} \gamma u_{tt}z d\Omega dt - \int_\Omega z_tzd\Omega \biggr\rvert_s^{T-s} + \int_s^{T-s}\!\!\!\!\int_{\Omega} fz d\Omega dt. \label{zmul} 
\end{align}
Identity \eqref{zmul} implies the following upper estimate for the kinetic energy
\begin{equation}\begin{aligned}\label{kinz}
    \int_s^{T-s}\!\!\!\!\int_{\Omega}|z_t|^2 d\Omega dt 
    &\lesssim [E_1(s)+E_1(T-s)] + \int_s^{T-s}\!\!\!\!\int_{\Omega}|\nabla{z}|^2 d\Omega dt + \int_s^{T-s}\!\!\!\!\int_{\Gamma}z\partial_\nu z d\Gamma dt
      \\
   & +\int_{0}^T D_\Psi(s)ds
    + \int_Q f^2 dQ
    + \|z\|_{L^2(s,T-s;L^2(\Omega))}^2
\end{aligned}\end{equation} and then
combining \eqref{potz} and \eqref{kinz} [adding and subtracting the boundary term  defining  $\|A^{1/2} z\|^2 $ from the gradient],  and noticing that the term $\|\gamma^{1/2}u_t\|$ in $E_1(t)$ can be obtained by the first two, we conclude
\begin{align}\label{e1est}
    \int_s^{T-s} E_1(t) dt
    &\lesssim E_1(s)+E_1(T-s) 
    + \int_0^T \!\!\!\!D_\Psi(s)ds \nonumber \\ &+ \int_s^{T-s}\!\!\!\!\tilde B(\Gamma)(t)dt + \int_Qf^2dQ
    + \|z\|_{L^2(s,T-s;L^2(\Omega))}^2.
\end{align}
where 
\begin{equation}
    \label{bddt2} \tilde B(\Gamma) := \dfrac{1}{2}\int_{\Gamma} \left(z_t^2 - b|\nabla z|^2\right)(h\cdot\nu) d\Gamma + b\int_{\Gamma} \partial_\nu zM_h(z) d\Gamma + \int_{\Gamma}z\partial_\nu z d\Gamma  + \kappa_0 \int_{\Gamma_0} |z|^2 d \Gamma_0
\end{equation}

We notice that for all the computations carried out so far we only needed the second part of our geometric assumption, that is, it would work with any $C^2$--vector field such that its Jacobian is strictly positive. Only now, in the analysis of the boundary term $\tilde B(\Gamma)$ will we use assumption \eqref{assc1} to prove the following lemma.

\begin{lemma}[\bf Key Lemmma]\label{bter}
   The boundary term $\tilde B(\Gamma)$ satisfies the following estimate \begin{align}
   \label{BEST} \int_s^{T-s} \tilde B(\Gamma)(t)dt &\lesssim E_1(s) + \varepsilon \int_s^{T-s}\|\nabla z\|_2^2dt \nonumber \\ &+ C_T\int_0^T D_\Psi(s)ds + C_T\|f\|_{H^{-1/2+\delta}(Q)}^2 + C_T lot_\delta(z),
   \end{align} where $lot_\delta(z)$ has the properties stated in Proposition \eqref{id}.
\end{lemma}\begin{proof}
        We need to estimate all the terms in \eqref{bddt2}. We immediately notice that  the first boundary term  contains  $|z_t|^2-|\nabla_{\tau} z|^2 $ evaluated on the boundary. However, on $\Gamma_0$ we have no information on either  tangential  derivative -- which provides  contribution unbounded with respect to the energy level. And it is at this point where we will be using orthogonality of the constructed vector field $h$ with respect to normal  to the boundary direction. The details are given below.  
        
         We start with the last term  in (\ref{bddt2}) . Notice that \begin{align}
           \int_s^{T-s} \int_\Gamma z\partial_\nu z d\Gamma dt &= \int_s^{T-s} \int_{\Gamma_0} z\left(-\dfrac{\kappa_0(x)}{\lambda} z\right) d\Gamma_0 dt + \int_s^{T-s} \int_{\Gamma_1} z\left(-\kappa_1(x) z_t\right) d\Gamma_1 dt \nonumber \\ & =  -\dfrac{1}{\lambda}\int_s^{T-s}\int_{\Gamma_0}\kappa_0 z^2d\Gamma_0 - \dfrac{1}{2}\int_{\Gamma_1}\kappa_1(x)z^2(T-s) d\Gamma_1 + \dfrac{1}{2}\int_{\Gamma_1}\kappa_1(x)z^2(s) d\Gamma_1  \nonumber \\& \leqslant \dfrac{1}{2}\int_{\Gamma_1}\kappa_1(x)z^2(s) d\Gamma_1 \lesssim \|A^{1/2} z(s)\|_2^2 \lesssim E_1(s), \label{aux7}
        \end{align} due to trace inequality. Next, we notice that \begin{align}
            \int_s^{T-s}\int_\Gamma \left(z_t^2 - b|\nabla z|^2\right)(h\cdot\nu) d\Gamma dt &= \int_s^{T-s}\left(\int_{\Gamma_0} + \int_{\Gamma_1}\right)\left(z_t^2 - b|\nabla z|^2\right)(h\cdot\nu) d\Gamma dt \nonumber \\& \lesssim \int_0^T\!\!\!\!D_\Psi(s)ds -b \int_s^{T-s} \int_{\Gamma_1}|\nabla z|^2(h \cdot \nu)d\Gamma dt, \label{aux4}
        \end{align} due to $h \cdot \nu = 0$ on $\Gamma_0$ and the definition of the damping term \eqref{damp}. \begin{remark}\label{rmkgeo}
        	Notice that assuming  only $h \cdot \nu \leqslant 0$,  would allow to dispense with the term $$\int_s^{T-s} \int_{\Gamma_0} z_t^2(h \cdot \nu) \leqslant 0.$$
	However, the  gradient  term  $$\int_s^{T-s} \int_{\Gamma_0} |\nabla z|^2(h \cdot \nu) =  \int_s^{T-s} \int_{\Gamma_0} (|\partial_\nu z|^2 + |\partial_\tau z|^2)(h \cdot \nu)$$ where $\partial_\tau$ indicates derivative in the tangential direction, poses difficulties.   Boundary condition on $\Gamma_0$ provide good estimate 
	for  the first part. However, for the second no estimate is available unless  $z_t|_{\Gamma_0} $  is under control, which is given through the dissipation 
	or $h \cdot \nu = 0$.
        \end{remark} 
    Back to \eqref{aux4},  for $\Gamma_1$ a tangential--trace estimate is available. In fact, using an adaptation of Lemma 2.1 in \cite{lasiecka_lebiedzik_NA_2002}, which was obtained via microlocal analysis of the homogeneous case, and accounting for our non--homogeneity $-\gamma u_{tt} + f$ in \eqref{zprob} we obtain
    \begin{align}\label{tang1}
    \int_s^{T-s} \int_{\Gamma_1}|\partial_\tau{z}|^2d\Gamma_1 dt
    &\leqslant C_T \int_0^T \int_{\Gamma_1}(|\partial_\nu{z}|^2 + z_t^2)d\Gamma_1 dt \nonumber \\ &+C_T\left[\|\gamma u_{tt} + f\|_{H^{-1/2+\delta}(Q)}^2 + lot_\delta(z) \right] \nonumber \\ &\lesssim C_T\int_0^T D_\Psi(s)ds + C_T\|f\|_{H^{-1/2+\delta}(Q)}^2 + C_T lot_\delta(z),
    \end{align} with $lot_\delta(z)$ complying with the condition stated in Proposition \ref{id}. With this we then improve estimate \eqref{aux4} as follows  \begin{align}
    \int_s^{T-s}\int_\Gamma \left(z_t^2 - b|\nabla z|^2\right)(h\cdot\nu) d\Gamma dt & \lesssim \int_0^T\!\!\!\!D_\Psi(s)ds + \int_s^{T-s} \int_{\Gamma_1}|\partial_\tau z|^2d\Gamma dt \nonumber \\ &\lesssim C_T\int_0^T D_\Psi(s)ds + C_T\|f\|_{H^{-1/2+\delta}(Q)}^2 + C_T lot_\delta(z), \label{aux5}
    \end{align} due to \eqref{tang1}. Finally, we tackle the more involved term. We notice that, \begin{align}
  \int_s^{T-s}\int_\Gamma \partial_\nu z M_h(z)d\Gamma dt &= \int_s^{T-s} \int_\Gamma \partial_\nu z \left(h\cdot \nabla z + \dfrac{1}{2} z\dvv{h}\right)d\Gamma dt \label{aux9} \\ &\lesssim \int_s^{T-s} \int_\Gamma \partial_\nu z \left(h\cdot \nabla z \right)d\Gamma dt + E_1(s), \label{aux8}
    \end{align} where we have used the fact that the second integral in \eqref{aux9} is exactly the one in \eqref{aux7} up to an uniformly bounded term. For the first integral in \eqref{aux9}, we use the identity $$\partial_\nu (h \cdot \nabla z) = |\partial_\nu z|^2(h\cdot \nu) + \partial_\nu z \partial_\tau z (h \cdot \tau)$$ which is obtained by writing the coordinates of the vector $\nabla z$ in the basis $\{\tau, \nu\}$. This allows us to write, recalling the damping terms \eqref{damp} and the tangential trace inequality \eqref{tang1}: \begin{align}
    \label{aux10} \int_s^{T-s} \!\!\!\!\int_\Gamma \partial_\nu z \left(h\cdot \nabla z \right)d\Gamma dt &\lesssim  \int_s^{T-s} \!\!\!\! \int_{\Gamma_0} \partial_\nu z \partial_\tau z (h \cdot \tau) \nonumber \\ &+ C_T\left[\int_0^T \!\!\!\!D_\Psi(s)ds + \|f\|_{H^{-1/2+\delta}(Q)}^2 + lot_\delta(z)\right]
    \end{align} and we now invoke Sobolev Embedding's Theory. Recall that $$\lambda\partial_\nu z = -\kappa_0 z \in H^{3/2}(\Gamma_0) \hookrightarrow H^{\delta_1}(\Gamma_0)$$ ($\delta_1 \leqslant 3/2$) since $z \in H^2(\Omega)$ if $\Psi$ is a classic solution. On the other hand, 
     $$\partial_\tau z \in H^{1/2}(\Gamma_0) \hookrightarrow H^{\delta_2}(\Gamma_0)$$ ($\delta_2 \leqslant 1/2$). Taking any $\delta_1 = \delta \in (0, 1/2]$ and $\delta_2 = - \delta_1 \in [-1/2,0)$ we have, by duality pairing along with continuity of the operator $\partial_\tau\rvert_{\Gamma_0}: H^{3/2 - \delta}(\Omega) \to H^{-\delta}(\Gamma_0)$ \begin{align}
    \int_s^{T-s}\int_{\Gamma_0}\partial_\nu{z}\partial_{\tau}{z}(h\cdot\tau) d\Gamma_0 dt
    &\lesssim \int_s^{T-s}\left[C_\varepsilon\|\partial_\nu{z}\|_{H^{\delta}(\Gamma_0)}^2 + \varepsilon\|\partial_\tau{z}\|_{H^{-\delta}(\Gamma_0)}^2 \right]dt \nonumber \\ &\lesssim  \int_s^{T-s}\left[C_\varepsilon\|z\|_{H^{\delta}(\Gamma_0)}^2 + \varepsilon\|z\|_{H^{3/2-\delta}(\Omega)}^2 \right]dt \nonumber \\ 
    &\lesssim C_\varepsilon \int_s^{T-s} \|z\|_{H^{1/2+\delta}(\Omega)}^2dt + \varepsilon \int_s^{T-s} \|\nabla z\|_{L^2(\Omega)}^2dt \nonumber \\ 
    &\lesssim \varepsilon \int_s^{T-s} \|\nabla z\|_{L^2(\Omega)}^2dt + C_{T,\epsilon}\|z\|_{H^{1/2+\delta}(\Omega)}^2 \label{aux11}, 
    \end{align}  finishing  the proof of Lemma \ref{bter}.
\end{proof}Finally, Lemma \ref{bter}  yields Proposition \ref{id} after taking $\varepsilon$ small enough. 
\end{proof}

Our next result aims at improving Lemma \ref{id} by absorbing $lot_\delta(z)$ by the damping. For the linear problem the compactness uniqueness argument used for achieving it  is 
stated below.
\begin{proposition}\label{compuniq}
For $T>0$ there exists a constant $ C_T > 0 $ such that  the following inequality holds:
\begin{equation}
    lot_\delta(z)
    \leq C_T\int_0^T D_{\Psi}(s)ds 
       \end{equation}
\end{proposition}

\begin{proof}
As pointed out in the statement of Proposition \eqref{id}, we have $$lot_\delta(z)\leq C_\delta\sup_{t\in[0,T]}\left\{\|z\|_{H^{1-\delta}(\Omega)}^2+\|z_t\|_{H^{-\delta}(\Omega)}^2\right\}$$ for $\delta \in (0,1/2)$. Then we prove Proposition \eqref{compuniq} as a corollary of the following Lemma \begin{lemma}
    \label{compunilem} There exists a constant $C_T$ such that \begin{equation}
        \label{cu1} \|(z,z_t)\|_{L^2(0,T; H^{1-\delta}(\Omega) \times H^{-\delta}(\Omega))}^2 \leqslant C_T\int_0^T D_{\Psi}(s)ds
    \end{equation}
     \end{lemma}
     \ifdefined\xxxxx
     As pointed out in \eqref{e1rec}, we have $lot_\delta(z)\leq C_\delta\sup_{t\in[0,T]}\left\{\|z\|_{H^{1-\delta}(\Omega)}^2+\|z_t\|_{H^{-\delta}(\Omega)}^2\right\}$, for $\delta \in (0,1/2)$. Then we prove Proposition \eqref{compuniq} as a corollary of the following Lemma \begin{lemma}
			\label{compunilem} There exists a constant $C_T$ such that \begin{equation}
			\label{cu1} \|(z,z_t)\|_{L^2(0,T; H^{1-\delta}(\Omega) \times H^{-\delta}(\Omega))}^2 \leqslant C_T\left[ b\int_{\Sigma_1} \kappa_1|z_t|^2 d\Sigma_1 
			+ \int_Q \gamma|u_{tt}|^2 dQ \right]
			\end{equation}
		\end{lemma}
		\fi
		\begin{proof}
	The proof is based on compactness-uniqueness argument.  Compactness follows from  compactness of Sobolev's embeddings  implicated in the definition of  of lower order terms  with respect to  the finite energy space for variables $(z,z_t) $ which are $H^1(\Omega) \times L_2(\Omega) $. 
	Uniqueness, instead follows  from the overdetermination of the 	wawe equation with overdetermined Neuman-Dirichlet  data on the boundary $\Gamma_1$. 	Using the notation of \cite{simon_AMOA_1986}, let $X = H^1(\Omega)$, $B = H^{1-\delta}(\Omega)$ and $Y = H^{-\delta}(\Omega).$ Then it follows from \cite[Theorem 16.1]{lions_magenes_SV_1972} that the injection of $X$ in $B$ is compact. Moreover, since $\delta \in (0,1/2)$, \cite[Theorem 12.4]{lions_magenes_SV_1972} allows us to write $$Y = H^{-\delta}(\Omega) = [L^2(\Omega),H^{-1}(\Omega)]_\delta,$$ and then the injection of $B$ in $Y$ is continuous (even dense). Introduce the space $\Lambda$ as $$\Lambda \equiv \{v \in L^2(0,T;X); \dot{v} \in L^2(0,T; Y)\}$$ equipped with the norm $$\|v\|_{\Lambda} = \|v\|_{L^2(0,T;X)} + \|\dot{v}\|_{L^2(0,T;Y)}.$$ Then it follows from \cite{simon_AMOA_1986} that the injection of $W$ into $L^2(0,T;B)$ is compact. We are then ready for proving \eqref{cu1}

			By contradiction, suppose that there exists a sequence of initial data $\{u_{0n},u_{1n},u_{2n}\}$ with corresponding $E_1^n(0)$ energy uniformly (in $n$) bounded generating  a sequence $\{u_n, \dot{u}_n, \ddot{u}_n\}$ of solutions of problem \eqref{usist} 
			with related 
			sequence $\left\{z_n = \dfrac{c^2}{b} u_n + \dot{u}_n, \dot{z}_n = \dfrac{c^2}{b}\dot{u}_n + \ddot{u}_n\right\}$ solutions of problem \ref{usistz} such that \begin{subnumcases}{}
			\|z_n\|_{L^2(0,T; H^{1 - \delta}(\Omega))}^2 +  \|\dot{z}_n\|_{L^2(0,T; H^{- \delta}(\Omega))}^2 \equiv 1 \label{E10-3}	 \\[2mm]
			\dfrac{c^2}{b} \int_0^T \int_\Omega \gamma(\ddot{u}_n)^2 dQ + \int_0^T \int_{\Gamma_1}\kappa_1 (\dot{z}_n)^2d\Sigma_1 \to 0, \ \mbox{as} \ n \to +\infty. \label{E10-4}
			\end{subnumcases}
			
			From idenity \eqref{e1id} (with $f = 0$) we see that the uniform boundedness $E_1^n(0)$ implies uniform boundedness of $E_1^n(t)$, $ t \in [0,T]$. Therefore, one might choose a (non--relabeled) subsequence satisfying\begin{subequations}\begin{align}
				z_{n} \to \ \mbox{some } \zeta, \ \mbox{weak}^\ast \ \mbox{in } L^\infty(0,T;H^1(\Omega)) 
				\label{E10-72a} \\ \dot{z}_{n} \to \ \mbox{some } \zeta_1, \ \mbox{weak}^\ast \ \mbox{in } L^\infty(0,T; L^2(\Omega)) \hookrightarrow L^2(0,T; H^{ - \delta}(\Omega)); \label{E10-72b} \\ 
				\gamma^{1/2}\dot{u}_{n} \to \ \mbox{some } \eta, \ \mbox{weak}^\ast \ \mbox{in } L^\infty(0,T;L^2(\Omega)); \label{E10-72c} 
				\end{align}\end{subequations}  It easily follows from distributional calculus that $\dot{\zeta} = \zeta_1$ and, in the limit, the functions $\zeta$ and $\eta$ satisfy the equation \begin{subnumcases}{}
			\ddot{\zeta} = b\Delta \zeta -\gamma^{1/2} \dot{\eta} \hspace{7cm} \mbox{in} \ Q  \label{E10-9a} \\[2mm]
			\gamma^{1/2}\dot{\zeta} = \dfrac{c^2}{b}\eta + \dot{\eta} \label{E10-9aa} \\[2mm]
			\left[\dfrac{\partial \zeta}{\partial \nu}+ \kappa_1\dot{\zeta}\right]\biggr\rvert_{\Sigma_1} = 0 ; \qquad \left[\dfrac{\partial \zeta}{\partial \nu}+ \kappa_0\zeta\right]\biggr\rvert_{\Sigma_0} = 0. \label{E10-9c} 
			\end{subnumcases} plus respective initial data.
			
			It follows from the weak convergence that there exist $M$ independent of $n$ such that \begin{equation}
			\|(z_n, \dot{z}_n)\|_{L^\infty(0,T; H^{1}(\Omega) \times H^{-\delta}(\Omega))} = \|z_n\|_\Lambda \leqslant \ M,\label{E10-17}
			\end{equation} for all $n$. Then, by \emph{compactness} (of $\Lambda$ in $L^2(0,T; H^{1-\delta}(\Omega))$ there exists a subsequence, still indexed by $n$, such that \begin{equation}
			z_n \to \zeta \ \mbox{strongly in } L^2(0,T; H^{1 - \delta}(\Omega)). \label{E10-21}
			\end{equation} 
			
			Next we show that $\eta$ and $\zeta$  are zero elements. 
			\ifdefined\xxxx
			\begin{equation} \dfrac{c^2}{b} \int_0^T \int_\Omega \dot{\eta}^2 dQ +  \int_0^T\int_{\Gamma_1} \kappa_1(\dot{\zeta})^2  d\Sigma_1 = 0.\label{E10-2311}
			\end{equation} To this end, recall from \eqref{E10-4} that \begin{equation}
			\dfrac{c^2}{b} \int_0^T \int_\Omega \gamma(\ddot{u}_{n})^2 dQ +  \int_0^T\int_{\Gamma_1} (\dot{z}_n)^2  d\Sigma_1 \to 0, \qquad n \to \infty,\label{E10-231}
			\end{equation} 
			\fi
			Indeed, from (\ref{E10-4}) 
			we obtain that  
			$\gamma^{1/2}\ddot{u}_n \to 0$ in $L^2(0,T;L^2(\Omega) )$ and
			$\dot{z}_n\rvert_{\Gamma_1}  \to 0 $ in $ L^2(0,T;L^2(\Gamma_1)).$ This implies that
			$\dot{\eta} =0 $ and $\dot{\zeta}|_{\Gamma_1} =0$. 
			Indeed, the last claim follows from 
			${\gamma^{1/2}} \ddot{u}_n \rightarrow \dot{\eta}$ in $H^{-1} (0,T;L^2(\Omega) $
			where by the uniqueness of the limit one must have $\dot{\eta} \equiv 0$ .
			Similar argument applies to infer $\dot{\zeta}|_{\Gamma_1} =0.$
			
			Next, passing to the limit as $n \to \infty$  yields the following over determined (on $\Gamma_1$) problem:
			\begin{subnumcases}{}
			\ddot{\zeta} = b\Delta \zeta  \hspace{7cm} \mbox{in} \ Q  \label{E10-9a2} \\[2mm]
			\gamma^{1/2}\dot{\zeta} = \dfrac{c^2}{b}\eta  \label{E10-9aa2} \\[2mm]
			\left[\dfrac{\partial \zeta}{\partial \nu}\right]\biggr\rvert_{\Sigma_1} = 0 ; \qquad \left[\dfrac{\partial \zeta}{\partial \nu}+ \kappa_0\zeta\right]\biggr\rvert_{\Sigma_0} = 0; \qquad \dot{\zeta}_t\rvert_{\Gamma_1} =0 \label{E10-9c2}
			\end{subnumcases} plus respective initial data.
			\ifdefined\xxxxx
			
			Recall that $\gamma^{1/2}\dot{u}_n \to \eta$ weak--star in $L^\infty(0,T;L^2(\Omega))$ and that $z_n \to \zeta$ strongly in $L^2(0,T; H^{1 - \delta}(\Omega))$ by \eqref{E10-21}, the last implying (from continuity of the trace operator) that $z_n\rvert_{\Gamma_1} \to \zeta\rvert_{\Gamma_1}$ strongly in $L^2(0,T; L^2(\Gamma_1))$. Moreover, it follows from \eqref{E10-231} that, at least for $n$ large, $\{\gamma^{1/2}\dot{u}_n,\dot{z}_n\rvert_{\Gamma_1}\}$ belongs to a fixed finite ball in $ L^2(0,T;L^2(\Omega)\times L^2(\Gamma_1))$ for all such $n.$
			
			Then, by maybe restricting to a further subsequence we have \begin{equation}
			\label{nv925} \gamma^{1/2}\ddot{u}_n \to \dot{\eta} \ \mbox{weakly in} \ L^2(0,T; L^2(\Omega)) \end{equation} \begin{equation}
			\label{nv926} \dot{z}_n\rvert_{\Gamma_1} \to \dot{\zeta}\rvert_{\Gamma_1} \ \mbox{weakly in} \ L^2(0,T; L^2(\Gamma_1)) \end{equation}
			In fact, if $\varphi \in \calD(0,T; \Omega)$ we have by definition of distributional derivative: \begin{equation}\label{nv927}
			\int \gamma^{1/2}\ddot{u}_n \varphi dQ = -\int \gamma^{1/2}\dot{u}_n \dot{\varphi}dQ \to - \int \eta \dot{\varphi}dQ = \int \dot{\eta}\varphi dQ
			\end{equation} recalling $\eta_n \to \eta$ weak--star in $L^\infty(0,T; L^2(\Omega))$ and the uniqueness of the limits. Similarly, if $\varphi \in \calD(0,T; \Gamma_1)$ then \begin{equation}\label{nv927}
			\int \dot{z}_n\rvert_{\Gamma_1} \varphi d\Sigma_1 = -\int z_n \rvert_{\Gamma_1} \dot{\varphi}d\Sigma_1 \to - \int \eta\rvert_{\Gamma_1} \dot{\varphi}d\Sigma_1 = \int \dot{\eta}\varphi d\Sigma_1
			\end{equation}recalling that $z_n\rvert_{\Gamma_1} \to \zeta\rvert_{\Gamma_1}$ strongly in $L^2(0,T; L^2(\Gamma_1)).$ This, we invoke the weak convergence in \eqref{nv925} and \eqref{nv926} along with the weak lower semicontinuity of norms to conclude that
			
			\begin{align}
			0  \leqslant \dfrac{c^2}{b} \int_0^T \int_\Omega \dot{\eta}^2 dQ &+  \int_0^T\int_{\Gamma_1} (\dot{\zeta})^2  d\Sigma_1 = \int_0^T \left[\dfrac{c^2}{b}\|\dot{\eta}\|_{L^2(\Omega)}^2 + \|\dot{\zeta}\|_{L^2(\Gamma_1)}^2\right]dt \nonumber \\ & \leqslant \liminf\limits_{n \to \infty}\int_0^T \left[\dfrac{c^2}{b}\|\gamma^{1/2}\ddot{u}_n\|_{L^2(\Omega)}^2 + \|\dot{z}_n\|_{L^2(\Gamma_1)}^2\right]dt = 0. \label{nv929}
			\end{align} 
			
			which establishes \eqref{E10-2311}.
			Next, notice that \eqref{E10-2311} along with \eqref{E10-9a}--\eqref{E10-9c} imply \begin{equation}
			\begin{cases}
			\dot{\eta} \equiv 0 \ \mbox{in} \ (0,T] \times \Omega \\
			\dot{\zeta} \equiv 0 \ \mbox{in} \ \Sigma_1
			\end{cases} \Longrightarrow \begin{cases}
			\ddot{\zeta} = b\Delta \zeta \ \mbox{in} \ Q \\
			\dfrac{\partial \zeta}{\partial \nu}\biggr\rvert_{\Sigma_1} \equiv 0, \ \zeta \big |_{\Sigma_0} \equiv 0.
			\end{cases} \label{E10-24i}
			\end{equation}
			\fi
			The overdetermined $\zeta$--problem implies  in particular with $v\equiv \zeta_t $
			$$\ddot{v} = b \Delta v $$ with the overdetermined boundary conditions 	$$\dfrac{\partial v}{\partial  \nu}\biggr\rvert_{\Gamma_1} = 0; \qquad  v\rvert_{\Gamma_1} = 0$$
			which yields overdetermination of boundary  data on $\Gamma_1$ for the wave operator.
			Here we have used  contradiction assumption  \eqref{E10-4}. 
			This gives  $v \equiv 0 $, hence $\zeta_t \equiv  0$ and $\zeta_{tt} =0 $   distributionally .
			Using this information in (\ref{E10-9a}) yields
			$$\Delta \zeta =0; \qquad  \dfrac{\partial \zeta}{\partial \nu}\biggr\rvert_{\Gamma_1} =0;\qquad
			\left[\dfrac{\partial \zeta}{\partial \nu} + \kappa_0 \zeta\right]_{\Gamma_0}  =0.$$
			Standard elliptic estimate, along with $\kappa_0 > 0$  gives $\zeta \equiv 0 $ in  $Q.$
			
			Finally, weak$^\ast$ convergence of $\dot z_n$ in $L^\infty(0,T; L^2(\Omega))$ and the compacity of $L^2(\Omega)$ into $H^{-\delta}(\Omega)$ (see \cite[Theorem 16.1 with $s = 0$ and $\varepsilon = \delta$]{lions_magenes_SV_1972}) we have $\dot{z}_n(t) \to \dot{\zeta}(t)$ strongly in $H^{-\delta}(\Omega)$ for a.e. $t \in [0,T]$ and this allow us to compute (due to Lebesgue dominated convergence theorem): \begin{align*}\lim\limits_{n \to \infty} \|\dot z_n\|_{L^2(0,T; H^{-\delta}(\Omega)}^2 &= \lim\limits_{n \to \infty} \int_0^T\|\dot z_n(t)\|_{H^{-\delta}(\Omega)}^2dt =  \int_0^T\lim\limits_{n \to \infty}\|\dot z_n(t)\|_{H^{-\delta}(\Omega)}^2dt = \|\dot{\zeta}\|_{L^2(0,T; H^{-\delta}(\Omega)}^2 = 0\end{align*} since $\dot{\zeta} \equiv 0$ in $Q.$ 
			Then, passing with the limit as $n \to \infty$ in \eqref{E10-3} we have $$0 = \|\zeta\|_{L^2(0,T;H^{1-\delta}(\Omega))} = 1,$$ which is a contradiction. The Lemma is proved.
		\end{proof}
		Lemma \ref{compunilem}  implies in a straightforward way the result of the  proposition \ref{compuniq}.
	\end{proof}
	
\ifdefined\xxxxxxxx
\begin{proof}
  Using the notation of \cite{simon_AMOA_1986}, let $X = H^1(\Omega)$, $B = H^{1-\delta}(\Omega)$ and $Y = H^{-\delta}(\Omega).$ Then it follows from \cite[Theorem 16.1]{lions_non-homogeneous_1972} that the injection of $X$ in $B$ is compact. Moreover, since $\delta \in (0,1/2)$, \cite[Theorem 12.4]{lions_non-homogeneous_1972} allows us to write $$Y = H^{-\delta}(\Omega) = [L^2(\Omega),H^{-1}(\Omega)]_\delta,$$ and then the injection of $B$ in $Y$ is continuous (even dense). Introduce the space $\Lambda$ as $$\Lambda \equiv \{v \in L^2(0,T;X); \dot{v} \in L^2(0,T; Y)\}$$ equipped with the norm $$\|v\|_{\Lambda} = \|v\|_{L^2(0,T;X)} + \|\dot{v}\|_{L^2(0,T;Y)}.$$ Then it follows from \cite{simon1986compact} that the injection of $W$ into $L^2(0,T;B)$ is compact. We are then ready for proving \eqref{cu1}

 By contradiction, suppose that there exists a sequence of initial data $\{u_{0n},u_{1n},u_{2n}\}$ with corresponding $E_1^n(0)$ energy uniformly (in $n$) bounded generating  a sequence $\{u_n, \dot{u}_n, \ddot{u}_n\}$ of solutions of problem \eqref{usist} 
	\dddot{y}_n + \alpha \ddot{y}_n - c^2 \Delta y_n - b\Delta \dot{y}_n = 0 \hspace{3.62cm} \mbox{in} \ Q\label{E10-1a} \\[2mm]
	y_n\big|_{t=0}  = y_{0n}; \quad \dot{y}_n\big|_{t=0}  = y_{1n}; \quad \ddot{y}_n\big|_{t=0}  = y_{2n} \hspace{1.47cm} \mbox{in} \ \Omega    \label{E10-1b}   \\[2mm]
	\left[\dfrac{\partial y_n}{\partial \nu}+ \dot{y}_n\right]\biggr\rvert_{\Sigma_1} = 0 \ \mbox{in} \ \Sigma_1 = (0,T] \times \Gamma_1; \qquad y_n\big|_{\Sigma_0} = 0   \ \mbox{in} \ \Sigma_0 = (0,T] \times \Gamma_0 \label{E10-1c} 
\end{subnumcases} 
with related 
sequence $\left\{z_n = \dfrac{c^2}{b} u_n + \dot{u}_n, \dot{z}_n = \dfrac{c^2}{b}\dot{u}_n + \ddot{u}_n\right\}$ solutions of problem \ref{usistz} such that \begin{subnumcases}{}
	\|z_n\|_{L^2(0,T; H^{1 - \delta}(\Omega))}^2 +  \|\dot{z}_n\|_{L^2(0,T; H^{- \delta}(\Omega))}^2 \equiv 1 \label{E10-3}	 \\[2mm]
	\dfrac{c^2}{b} \int_0^T \int_\Omega \gamma(\ddot{u}_n)^2 dQ + \int_0^T \int_{\Gamma_1}\kappa_1 (\dot{z}_n)^2d\Sigma_1 \to 0, \ \mbox{as} \ n \to +\infty. \label{E10-4}
\end{subnumcases}

From idenity \eqref{e1id} (with $f = 0$) we see that the uniform boundedness $E_1^n(0)$ implies uniform boundedness of $E_1^n(t)$, $ t \in [0,T]$. Therefore, one might choose a (non--relabeled) subsequence satisfying \begin{subequations}\begin{align}
		z_{n} \to \ \mbox{some } \zeta, \ \mbox{weak}^\ast \ \mbox{in } L^\infty(0,T;H^1(\Omega)) 
		\label{E10-72a} \\ \dot{z}_{n} \to \ \mbox{some } \zeta_1, \ \mbox{weak}^\ast \ \mbox{in } L^\infty(0,T; L^2(\Omega)) \hookrightarrow L^2(0,T; H^{ - \delta}(\Omega)); \label{E10-72b} \\ 
	\gamma^{1/2}\dot{u}_{n} \to \ \mbox{some } \eta, \ \mbox{weak}^\ast \ \mbox{in } L^\infty(0,T;L^2(\Omega)); \label{E10-72c} 
\end{align}\end{subequations}  It easily follows from distributional calculus that $\dot{\zeta} = \zeta_1$ and, in the limit, the functions $\zeta$ and $\eta$ satisfy the equation \begin{subnumcases}{}
\ddot{\zeta} = b\Delta \zeta -\gamma^{1/2} \dot{\eta} \hspace{7cm} \mbox{in} \ Q  \label{E10-9a} \\[2mm]
	\gamma^{1/2}\dot{\zeta} = \dfrac{c^2}{b}\eta + \dot{\eta} \label{E10-9aa} \\[2mm]
	\left[\dfrac{\partial \zeta}{\partial \nu}+ \kappa_1\dot{\zeta}\right]\biggr\rvert_{\Sigma_1} = 0 ; \qquad \left[\dfrac{\partial \zeta}{\partial \nu}+ \kappa_0\zeta\right]\biggr\rvert_{\Sigma_0} = 0. \label{E10-9c} 
\end{subnumcases} plus respective initial data.

It follows from the weak convergence that there exist $M$ independent of $n$ such that \begin{equation}
		\|(z_n, \dot{z}_n)\|_{L^\infty(0,T; H^{1}(\Omega) \times H^{-\delta}(\Omega))} = \|z_n\|_\Lambda \leqslant \ M,\label{E10-17}
	\end{equation} for all $n$. Then, by \emph{compactness} (of $\Lambda$ in $L^2(0,T; H^{1-\delta}(\Omega))$ there exists a subsequence, still indexed by $n$, such that \begin{equation}
		z_n \to \zeta \ \mbox{strongly in } L^2(0,T; H^{1 - \delta}(\Omega)). \label{E10-21}
	\end{equation} 

Next we show that $\eta$ and $\zeta$  are zero elements. 
Indeed, from (\ref{E10-4}) 
we obtain that  
$\gamma^{1/2}\ddot{u}_n \to 0$ in $L^2(0,T;L^2(\Omega) $ and
$\dot{z}_n\rvert_{\Gamma_1}  \to 0 $ in $ L^2(0,T;L^2(\Gamma_1)).$ This implies that
 $\dot{\eta} =0 $ and $\dot{\zeta}|_{\Gamma_1} =0$. 
\end{proof}

\fi

We are ready to establish the exponential decay of the the energy functional $E_1$.
\begin{theorem}\label{thm34}
Assume that $f=0$. Hence, the energy functional $E_1$ is exponentially stable, i.e. there exists $T>0$ and constants $M,\omega>0$ such that 
\begin{equation}\label{e1exp}
    E_1(t) \leq M e^{-\omega t} E_1(0),\quad\mbox{for}~t\geqslant 0.
\end{equation}
\end{theorem}
\begin{proof}
For $f = 0$, identity \eqref{e1id} implies
\begin{equation*}
\left(\int_0^s + \int_{T-s}^T\right) E_1(t) dt \leq 2sE_1(0).
\end{equation*}
Since $s<T/2$ can be taken arbitrarily small, we fix $s<1/2$ in the above inequality. Then by dissipativity of $E_1$ (for $f = 0$) along with Propositions \ref{id} and \ref{compuniq} we infer
\begin{equation}\label{e1xp_1}
    \int_0^T E_1(t) dt
    \lesssim E_1(T) + C_T\int_0^T D_\Psi(s)ds.
\end{equation}
On the other hand, using identity \eqref{e1id} (with $f=0$) once more, we deduce
\begin{equation}\begin{aligned}\label{e1xp_2}
    TE_1(T)
    \lesssim \int_0^TE_1(t)dt + C_T\int_0^T D_\Psi(s)ds.
\end{aligned}\end{equation}
Combining \eqref{e1xp_1} and \eqref{e1xp_2} we arrive at
\begin{equation*}
    (T-C)E_1(T) + \int_0^T E_1(t) dt
    \leq C_T \int_0^T D_\Psi(s)ds
\end{equation*}
for some $C>0$. Choosing $T=2C$ and replacing the ``damping'' term using identity \eqref{e1id} (with $f=0$) we rewrite the above estimate as follows
\begin{equation*}
    E_1(T) + \int_0^T E_1(t) dt
    \lesssim C_T[E_1(0) - E_1(T)]
\end{equation*}
which implies
\begin{equation*}
    E_1(T) \leqslant \frac{C_T}{1+C_T} E_1(0) = {\mu} E_1(0),
\end{equation*}
where $0<\mu<1$ does not depend on the solution. This implies
\eqref{e1exp} with $\omega = |\ln{\mu}|/T$ and $M=1/\mu$.
\end{proof}

The result of Theorem \ref{thm34}  is the  key to establish the exponential stability of of the semigroup $S(t)$, generated by $\mathbb{A}$ on  $\mathbb{H}.$

\subsection{Proof of Theorem \ref{thm34n} }
\begin{proof}[\bf Part (i)] Notice that the exponential decay for $E_1$ obtained in Theorem \eqref{thm34} implies exponential decay of the quantities $\|z\|_{\calD(A^{1/2})}, \|z_t\|_{L^2(\Omega)}$, and we will show that this implies exponential decay for the total energy  $E(t)$, provided that the initial data $u_0 $ is controlled with respecyt to the topology induced by $A^{1/2} $. For this, the only remaining quantity we need to show exponential decay is $\| u\|_{\calD(A^{1/2})}$ and this follows from the fact that $bu_t + c^2u = z$. Indeed, the variation of parameters formula implies that \begin{equation}
		\label{NE4-35} u(t) = e^{-\frac{c^2}{b}t}u_0 + \int_0^t e^{-\frac{c^2}{b}(t-\tau)}z(\tau)d\tau,
	\end{equation} then, computing the $\calD(A^{1/2})$--norm both sides we estimate \begin{equation}
		\|u(t)\|_{\calD(A^{1/2})} \leqslant e^{-\frac{c^2}{b} t} \|u_0\|_{\calD(A^{1/2})} + \int_0^t e^{-\frac{c^2}{b}(t-\tau)}\|z(\tau)\|_{\calD(A^{1/2})}d\tau \label{E77nn-15}
	\end{equation} hence it follows from \eqref{e1exp} that \begin{align*}
		\|u(t)\|_{\calD(A^{1/2})} &\leqslant e^{-\frac{c^2}{b} t} \|u_0\|_{\calD(A^{1/2})} + ME_1(0)\int_0^t e^{-\frac{c^2}{b}(t-\tau)-\omega\tau}d\tau \nonumber \\ &\leqslant e^{-\frac{c^2}{b} t} E(0) + \dfrac{(c^2 -b\omega)(e^{-\omega t}-e^{-\frac{c^2}{b}t})}{\omega c^2}ME(0)\leqslant \overline{M}e^{-\overline{\omega}t}E(0).
	\end{align*} where we have made the benign assumption that $\dfrac{c^2}{b} > \omega$ from \eqref{e1exp}, as if $\omega \geqslant \dfrac{c^2}{b}$ we use formula \eqref{e1exp} with $\omega_1 := \dfrac{c^2}{b} - \varepsilon$ so $\omega > \omega_1$ and $\dfrac{c^2}{b} > \omega_1$.
	
	The proof is complete.
\end{proof}

\begin{proof} [\bf Part (ii)]
The first step towards showing $\mathbb{H}_1$--level stabilization is to derive energy estimate for the higher order energy functional $E_2$. We start with a basic multiplier identity. \begin{proposition}
	Let $\Psi = (u,z,z_t)$ be a classical solution for \eqref{zprob}. Then for all $0 \leqslant s < t \leqslant T$ the following identity holds \begin{align}
	\label{deltaz} 
	b\int_s^t (\Delta z, \Delta u)d\sigma &= \left[(z_t,\Delta u)+\dfrac{1}{2}\|\nabla{z}\|_2^2\right]\biggr\rvert_s^t+\int_s^t\int_\Gamma z_t \partial_\nu zd\Gamma d\sigma \nonumber \\ &+\int_s^t\left[\left(\dfrac{c^2}{b}z_t+\gamma u_{tt},\Delta u\right)\right]d\sigma -\int_s^t(f,\Delta u)d\sigma
	\end{align}
\end{proposition} \begin{proof}We start by noticing that since $\Psi$ is classical, we have $\Delta u, \Delta u_t \in L^2(\Omega)$, moreover, we compute \begin{align*}
(z_{tt},\Delta u) &= \dfrac{d}{dt}(z_t,\Delta u) - (z_t,\Delta u_t) = \dfrac{d}{dt}(z_t,\Delta u) - \left(z_t,\Delta \left(z-\dfrac{c^2}{b}u\right)\right) \\&= \dfrac{d}{dt}(z_t,\Delta u) + \dfrac{1}{2}\dfrac{d}{dt}\|\nabla z\|_2^2 -\dfrac{1}{\lambda}\int_{\Gamma_0}z_t(\kappa_0(x)z)\Gamma_0 -
\int_{\Gamma_1}z_t(\kappa_1(x)z_t)d\Gamma_1+\left(\dfrac{c^2}{b}z,\Delta u\right) \\&= \dfrac{d}{dt}(z_t,\Delta u) + \dfrac{1}{2}\dfrac{d}{dt}\|\nabla z\|_2^2 -\dfrac{1}{2\lambda}\dfrac{d}{dt}\|\kappa_0^{1/2}z\|_{\Gamma_0}^2 -
\|\kappa_1^{1/2}z_t\|_{\Gamma_1}^2+\left(\dfrac{c^2}{b}z,\Delta u\right).
\end{align*}
	Thus, taking the $L^2$--inner product of $z_{tt} - b\Delta z = -\gamma u_{tt} + f$ with $\Delta u \in L^2(\Omega)$ we get \begin{align*}
	b\int_s^t(\Delta z,\Delta u)d\sigma &= \int_s^t (z_{tt} +\gamma u_{tt}-f,\Delta u)d\sigma \nonumber \\ &=\int_s^t (z_{tt} , \Delta u)d\sigma +\int_s^t (\gamma u_{tt}-f,\Delta u)d\sigma \nonumber \\ &= \left[(z_t,\Delta u)+\dfrac{1}{2}\|\nabla z\|_2^2-\dfrac{1}{2\lambda} \|\kappa_0^{1/2}z\|_{\Gamma_0}^2\right]\biggr\rvert_s^t-\int_s^t \|\kappa_1^{1/2}z_t\|_{\Gamma_1}^2 d\sigma\nonumber \\ &+\int_s^t\left[\left(\dfrac{c^2}{b}z_t+\gamma u_{tt},\Delta u\right)\right]d\sigma -\int_s^t(f,\Delta u)d\sigma.
	\end{align*} 
\end{proof}

We now derive the estimate for $E_2.$ We take the $L^2$--inner product of $$\Delta(bu_t + c^2 u) = b\Delta z$$ with $\Delta u$ and integrate in time. By connecting it with \eqref{deltaz} one obtains that  \begin{align}\label{eq0}b\|\Delta u\|_2^2 + c^2\int_0^T \|\Delta u\|_2^2 &= b\|\Delta u_0\|_2^2 + 
b\int_0^T (\Delta z, \Delta u) \nonumber  \\ &= b\|\Delta u_0\|_2^2 + \left[(z_t,\Delta u)+\dfrac{1}{2}\|\nabla{z}\|_2^2-\dfrac{1}{2\lambda} \|\kappa_0^{1/2}z\|_{\Gamma_0}^2\right]\biggr\rvert_0^T\nonumber \\ &-\int_0^T \|\kappa_1^{1/2}z_t\|_{\Gamma_1}^2 d\sigma+\int_0^T\left[\left(\dfrac{c^2}{b}z_t+\gamma u_{tt},\Delta u\right)\right]d\sigma -\int_0^T(f,\Delta u)d\sigma\end{align} 

Now since all terms in \eqref{eq0} are benign in the sense that all (but $f$ and $\Delta u$) are either members of $E(t)$ or are bounded above by the damping, it follows that for each $\varepsilon > 0$ there exists $C_{\varepsilon} > 0$ such that \begin{align} \label{eq1}
b\|\Delta u(T)\|_2^2 + c^2\int_0^T \|\Delta u\|_2^2 &\lesssim \E(0) + \varepsilon \left(\|\Delta u\|_2^2 + \int_0^T \|\Delta u\|_2^2 d\sigma\right) \nonumber \\&+ C_{\varepsilon} \left(E_1(t) + \int_0^T E_1(\sigma)d\sigma +\int_Q f^2dQ\right).
\end{align}

Then, taking $\varepsilon$ small and using \eqref{e1exp} we have  \begin{align} \label{eq11}
b\|\Delta u(T)\|_2^2 + c^2\int_0^T \|\Delta u\|_2^2 &\lesssim \E(0) + \int_0^T E_1(\sigma)d\sigma +\int_Q f^2dQ,
\end{align}
\ifdefined\xxxxxxxxx
and from here we get exponential stability of the semigroup $T(t)$ generated by $\mathbb{A}$ in $\mathbb{H}_1$ as we explain below.

{\bf Marcelo: I will be adding some material here just to make sure that we are on the same page. In addition, I will be very busy later on-trying to expediate matters as much as possible}. 
{\bf Marcelo: Here you need also to generate the remaining parts of $H^2 $ norm-normal derivatives in $H^{1/2} $ . You use invariance of the compatibility conditions  under the dynamics, and then regularity of the normals and then elliptic tehory -or refer to the definition of the norm in $H_1$. }. 
\fi

From \eqref{eq11} we have obtained that $\Delta u(t) \in L^2(\Omega) $. In addition $ (u,u_t,u_{tt}) \in \mathbb{H} $ implies that $u(t) \in H^1(\Omega) $ and $u_t(t) \in H^1(\Omega) $. 
By a  standard duality argument one this  obtains that $\partial_{\nu } u(t)  \in H^{-1/2}(\Gamma) $. We will be able to improve this regularity by appealing to $\mathbb{H} $ regularity already obtained in te previous section. 
On the other hand, by using invariance of boundary conditions we also have 
$$ \partial_{\nu} u(t)|_{\Gamma_0}  = -\kappa_0 u(t) \in H^{1/2}(\Gamma_0 ) \qquad  \partial_{\nu} u(t)|_{\Gamma_1}  = -\kappa_1 u_t(t) \in H^{1/2}(\Gamma_1).$$
By the definition of the norm in $\mathbb{H}_1$ the above implies that $ (u,u_t,u_{tt}) \in \mathbb{H}_1 $, as desired. Moreover we have a control of the norms: 
\begin{align*}\|(u,u_t,u_{tt})\|_{\mathbb{H}_1} &\leqslant  C| |(u,u_t,u_{tt})\|_{\mathbb{H}} + \|\Delta u(t)\|_2 +  \|\sqrt{\kappa_0}u(t)\|_{H^{1/2} (\Gamma_0 } +  \|\sqrt{\kappa_1}u_t(t)|_{H^{1/2}(\Gamma_1)} \\ &
\leqslant C\left(\|(u,u_t,u_{tt})\|_{\mathbb{H}} + \|\Delta u(t)\|_2\right) \end{align*}
which proved the desired regularity in $\mathbb{H}_1$. We are ready to complete the proof. 

Let $f = 0.$ Adding $E(T)  + \int_0^T E(\sigma) d\sigma$ to both sides of  \eqref{eq11} we obtain, \begin{align*}
	\E(T) + \int_0^T \E(\sigma)d\sigma &\lesssim \E(0) + E(T) + \int_0^T E(\sigma)d\sigma \\ &\leqslant \E(0) + ME(0)e^{-\omega t} + ME(0)\int_0^t e^{-\omega \sigma}d\sigma \\ &= \E(0)  + ME(0)e^{-\omega t} - \omega^{-1}ME(0)\left[e^{-\omega t}-1\right] < +\infty,
	\end{align*} for all $t \geqslant 0$, for some $\omega, M > 0.$ By making $T \to \infty$ we see that $$\int_0^\infty \E(\sigma)d\sigma < +\infty,$$ and the result follows by Pazy--Datko's Theorem \cite{pazy_S_1992}.
\end{proof}

\section{Proof of Theorem \ref{nsem} -- Construction of Global $\mathbb{H}_1$-- valued Solutions}\label{nonlinearsem}

Our goal now is to prove that fixed--point solutions can be constructed  for the nonlinear problem in $\mathbb{H}_1.$ To this end, fix $r > 0$ such that $\|\Phi_0\|_{\mathbb{H}_1} \leqslant r$ and let $X$ be the set defined as 
\begin{equation*}
X^\beta = \left\{\Psi = \begin{bmatrix}w \\ w_t \\ w_{tt}\end{bmatrix} \in C([0,T]; \mathbb{H}_1); \sup_{t \in [0,T]} \|\Psi(t)\|_{\mathbb{H}_1} \lesssim r+1 \ \mbox{and} \ \sup_{t \in [0,T]}\|\Psi(t)\|_{\mathbb{H}} < \beta \right\}
\end{equation*}
where $\beta > 0$ is for the time being a given positive number but we will take it to be sufficiently small later. Moreover, the condition $\sup\limits_{t \in [0,T]} \|\Psi(t)\|_{\mathbb{H}_1} \lesssim r +1$ simply means that solutions will exist in bounded sets of $C([0,T];\mathbb{H}_1)$ with respect to $\mathbb{H}$ but this introduces no restriction on the size of the data in $\mathbb{H}_1$. The number $1$ could, then, be replaced by any other positive number. Let's equip $X^\beta$ it with the norm 
\begin{equation*}\label{norm}\|\Psi\|_{X^\beta}^2 := \sup_{t \in [0,T]} \|\Psi(t)\|_{\mathbb{H}_1}^2.\end{equation*} We start with a regularity lemma. \begin{lemma}\label{fest}For $\Psi = (w,w_t,w_{tt})^\top$ let the action $\cF$ on $\Psi$ be given by \begin{equation}
	\label{eqF} \cF(\Psi) = \dfrac{1}{\tau}\begin{bmatrix}
	0 \\ 0 \\ w_t^2 + ww_{tt}
	\end{bmatrix}.
	\end{equation} Then the following assertions hold true: \begin{itemize}
		\item[\bf (i)] $\cF$ defines a continuous map $\cF: X^\beta \to C([0,T]; \mathbb{H}_1)$ and, in particular, for each $t$ the inequality \begin{equation}
		\label{contcF} \left\|\cF(\Psi(t))\right\|_{\mathbb{H}_1} \leqslant \dfrac{C\beta}{\tau}\|\Psi(t)\|_{\mathbb{H}_1}, \qquad \Psi \in X^\beta
		\end{equation} holds for some $C>0$ fixed.
		
		\item[\bf (ii)] Stronger than continuity, the following estimate holds:\begin{equation}\label{imp2}\tau\|\cF(\Phi)\|_{C([0,T];\mathbb{H}_1)} \lesssim \beta^2 + \beta^{3/2}\sqrt{r+1}.\end{equation} 
	\end{itemize} \end{lemma}
	 \begin{proof}Recall that $w_t \in H_{\Gamma_1}^1(\Omega) \hookrightarrow L^4(\Omega)$ and then $w_t^2 \in C([0,T]; L^2(\Omega)).$ Moreover, since $H^2(\Omega) \hookrightarrow L^\infty(\Omega)$ it follows that $ww_{tt} \in C([0,T];L^2(\Omega)).$ For each $t$, interpolation inequalities\footnote{$\|w\|_4 \lesssim \|w\|_2^{1/2}\|w\|_{H^1}^{1/2}$ for all $w \in H^1(\Omega)$}\footnote{$\|w\|_\infty \lesssim \|w\|_{H^1}^{1/2}\|w\|_{H^2}^{1/2}$ for all $w \in H^2(\Omega)$} give  \begin{align}
\tau\|\cF(\Phi)(t)\|_{\mathbb{H}_1} &= \|w_t^2 + ww_{tt}\|_2 \lesssim \|w_t\|_4^2 + \|w\|_{\infty}\|w_{tt}\|_2 \nonumber \\ &\lesssim  \left\|\p w_t\right\|_2^2 + \left\|\p w\right\|_2^{1/2}\|w\|_{H^2(\Omega)} ^{1/2}\|w_{tt}\|_2 \label{interm} \\ &\lesssim \|\Psi(t)\|_{\mathbb{H}}\|\Psi(t)\|_{\mathbb{H}_1} \lesssim \beta\|\Psi(t)\|_{\mathbb{H}_1} \nonumber\end{align} which yields \eqref{contcF} and, by taking the supremum on both sides, also (i) altogether. Moreover, returning to the intermediate estimate \eqref{interm}, we further notice \begin{align*}
	\tau\|\cF(\Phi)(t)\|_{\mathbb{H}_1} &\lesssim \left\|\p w_t\right\|_2^2 + \left\|\p w\right\|_2^{1/2}\| w\|_{H^2(\Omega)}^{1/2}\|w_{tt}\|_2 \\[2mm] &\lesssim \left[\sup_{t \in [0,T]}\|\Psi(t)\|_{\mathbb{H}}\right]^2 + \left[\sup_{t \in [0,T]}\|\Psi(t)\|_{\mathbb{H}}\right]^{3/2}\left[\sup_{t \in [0,T]}\|\Psi(t)\|_{\mathbb{H}_1}\right]^{1/2} \\[2mm]  &\lesssim \beta^2 + \beta^{3/2}\sqrt{\sup_{t \in [0,T]}\|\Psi(t)\|_{\mathbb{H}_1}},
	\end{align*}which yields \eqref{imp2} and completes the proof.\end{proof}
The validity of the previous Lemma along with the fact that $\mathcal{A}$ generates $C_0$--semigroups $T(t)$ and $S(t)$ on $\mathbb{H}_1$ and $\mathbb{H}$ respectively, guarantees that, for each $\Psi \in X$ there exists a unique $\Phi = (u,u_t,u_{tt})^\top =: \Theta(\Psi) \in C([0,T];\mathbb{H}_1)$ solution of \eqref{usist} characterized as the variation of parameters formula with forcing term $\cF(\Psi)$ and initial condition $\Phi_0 = (u_0,u_1,u_2) \in \mathbb{H}_1$, i.e., \begin{equation}
\label{varpar} \Theta(\Psi)(t) = T(t)\Phi_0 + \int_0^t T(t-\sigma)\cF(\Psi)(\sigma) d\sigma
\end{equation} and we note that the same formula is valid if we replace $T(t)$ by $S(t).$ Moreover, uniform exponential stability implies the existence of numbers $\omega_0,\omega_1, M_0,M_1 > 0$ such that \begin{equation}\label{exp}\|T(t)\Phi_0\|_{\mathbb{H}_1}  \leqslant M_1e^{-\omega_1 t}\|\Phi_0\|_{\mathbb{H}_1} \ \mbox{and} \ \|T(t)\Phi_0\|_{\mathbb{H}}  \leqslant M_0e^{-\omega_0 t}\|\Phi_0\|_{\mathbb{H}}\end{equation}for all $t \geqslant 0.$ Among other properties, the exponential stability of the linear problem implies invariance of the map $\Theta$ in $X^\beta$, as we make precise below.
\begin{lemma}. Given $\Phi_0 \in \mathbb{H}_1 $ such that  $\|\Phi_0\|_{\mathbb{H}_1}\leq C$ . Then,  there exist $\beta >0$ and $\rho_{\beta} > 0$ with the property that if $\|\Phi_0\|_{\mathbb{H}} < \rho_\beta$ then the map $\Theta$ is $X^\beta$--invariant. \end{lemma} \begin{proof}
	Proving this claim is equivalent to prove that there exists $\beta > 0$ for which $\|\Theta(\Psi)(t)\|_{\mathbb{H}_1} \lesssim r+1$ and $\|\Theta(\Psi)(t)\|_{\mathbb{H}} < \beta$ for all $t \in [0,T)$ and each $\Psi \in X^\beta$, provided $\|\Phi_0\|_{\mathbb{H}} < \rho_\beta$, with $\rho_\beta$ conveniently chosen.  From \eqref{varpar} and \eqref{exp} it follows, for each $t \in [0,T)$, \begin{align}
	\label{invar} \|\Theta(\Psi)(t)\|_{\mathbb{H}_1} &\leqslant  \|T(t)\Phi_0\|_{\mathbb{H}_1} + \int_0^t \left\|T(t-\sigma)\cF(\Psi)(\sigma)\right\|_{\mathbb{H}_1}d\sigma \nonumber \\ & \leqslant M_1\left(\|\Phi_0\|_{\mathbb{H}_1} + \int_0^t e^{-\omega_1 (t-\sigma)}\|\cF(\Psi)(\sigma)\|_{\mathbb{H}_1} d\sigma\right)\nonumber \\ &\lesssim M_1[ \|\Phi_0\|_{\mathbb{H}_1} + \dfrac{C_{\omega_1}}{\tau} \sup\limits_{t \in [0,T] }\tau\|\cF(
\Psi)(t)\|_{\mathbb{H}_1}]  \nonumber \\ &\lesssim   M_1 C  + M_1 \tau^{-1}C_{\omega_1}  \left(\beta^2 + \beta^{3/2}\sqrt{r+1}\right) \lesssim r+1,
	\end{align}  provided $M_1 C < 1/2(r+1) $  and  $\beta$  is sufficiently small. Moreover, by Lemma \ref{fest} (and again \eqref{varpar} and \eqref{exp}) \begin{align}
	\label{invar2} \|\Theta(\Psi)(t)\|_{\mathbb{H}} &\leqslant  \|T(t)\Phi_0\|_{\mathbb{H}} + \int_0^t \left\|T(t-\sigma)\cF(
	\Psi)(\sigma)\right\|_{\mathbb{H}}d\sigma \nonumber \\ & \leqslant M_0\left(\|\Phi_0\|_{\mathbb{H}} + \int_0^t e^{-\omega_0 (t-\sigma)}\|\cF(\Psi)(\sigma)\|_{\mathbb{H}_1} d\sigma\right)\nonumber \\ &\lesssim M_0  \|\Phi_0\|_{\mathbb{H}} + \dfrac{M_0 C_{\omega_0}}{\tau} \sup\limits_{t \in [0,T] }\tau\|\cF(\Psi)(t)\|_{\mathbb{H}_1} \nonumber \\ &\lesssim \rho_{\beta} + \left(\beta^2 + \beta^{3/2}\sqrt{r+1}\right) < \beta,
	\end{align} provided $\beta$ and $\rho_{\beta}< 1/2 \beta $ are sufficiently small. 
\end{proof}

We are then ready to prove that for a (possibly smaller) $\beta$, the map $\Theta$ is a contraction. \begin{lemma}There exist $\beta >0$ and $\rho_{\beta} > 0$ with the property that if $\|\Phi_0\|_{\mathbb{H}} < \rho_\beta$ then $\Theta$ is a contraction.\end{lemma} \begin{proof}
	Let $\Psi_1, \Psi_2 \in X^\beta$, $\Psi_1 = (v,v_t,v_{tt})^\top$ and $\Psi_2 = (w,w_t,w_{tt})^\top.$ The key point of this proof is to estimate $\|\cF(\Psi_1)-\cF(\Psi_2)\|_{C([0,T];\mathbb{H}_1)},$ which is where we start. First notice that, since the first two coordinates of both $\cF(\Psi_1)$ and $\cF(\Psi_2)$ are zero, we just care about the third one, whose difference, for each t, is given by \begin{align}
	 v_t^2 + vv_{tt} - w_t^2 - ww_{tt}  &= \underbrace{(v_t+w_t)(v_t-w_t)}_{=I_1(t)} + \underbrace{(v-w)v_{tt}}_{=I_2(t)} + \underbrace{w(v_{tt}-w_{tt})}_{=I_3(t)} \nonumber \\ &= I_1(t) + I_2(t) + I_3(t).
	\end{align}
	Now we estimate the supremmum of the $L^2$-norm of $I_1$. For this we notice that a combination of Holder's Inequality with the Sobolev Embedding $H_{\Gamma_1}^2(\Omega) \hookrightarrow L^4(\Omega)$ yields \begin{align*}
	\|I_1(t)\|_2 &= \left\|({v}_t+{w}_t)({v}_t-{w}_t)\right\|_{2} \leqslant \left(\|{v}_t\|_4+\|{w}_t\|_4\right)\|{v}_t-{w}_t\|_4 \nonumber \\ & \lesssim \left(\|\nabla{v}_t\|_2+\|\nabla{w}_t\|_2\right)\left\|\nabla({v}_t-{w}_t)\right\|_2 \lesssim \beta \|\Psi_1-\Psi_2\|_{X^\beta},
	\end{align*} for each $t$. Then $\sup\limits_{t \in [0,T]} \|I_1(t)\|_2 \leqslant \beta \|\Psi_1 - \Psi_2\|_{X^\beta}.$ Next, for estimating the suppremum of the $L^2$--norm of $I_2$ we notice that the sobolev embedding $H_{\Gamma_1}^2(\Gamma) \hookrightarrow L^\infty(\Omega)$ yields \begin{align*}
	\|I_2(t)\|_2 &= \|v_{tt}(v-w)\|_2 \leqslant \|v_{tt}\|_2 \|v-w\|_\infty \\ &\lesssim \|v_{tt}\|_2\|\Delta(v-w)\|_2 \lesssim \beta \|\Psi_1 - \Psi_2\|_{X^\beta},
	\end{align*} for each $t \in [0,T).$ Then $\sup\limits_{t \in [0,T]} \|I_2(t)\|_2 \leqslant \beta \|\Psi_1 - \Psi_2\|_{X^\beta}.$ Finally, for estimating the supremum of the $L^2$--norm of $I_3$ we will use the (further to the Sobolev emdedding $H_{\Gamma_1}^2(\Omega) \hookrightarrow L^\infty(\Omega)$) the interpolation inequality $\|w\|_\infty \lesssim \|\nabla w\|_2^{1/2}\|\Delta w\|_2^{1/2}$ which holds for all $w \in H_{\Gamma_1}^2(\Omega).$ We have \begin{align*}
	\|I_3(t)\|_2 &= \|w(v_{tt}-w_{tt})\|_2 \leqslant \|w\|_\infty \|v_{tt}-w_{tt}\|_2 \\ &\lesssim \|\nabla w\|_2^{1/2}\|\Delta w\|_2^{1/2}\|v_{tt}-w_{tt}\|_2 \\ &\lesssim \beta^{1/2} \sqrt{\sup\limits_{t \in [0,T]}\|\Psi_2(t)\|_{\mathbb{H}_1}} \ \|\Psi_1 - \Psi_2\|_{X^\beta} \lesssim \beta^{1/2} \sqrt{r+1} \|\Psi_1 - \Psi_2\|_{X^\beta},
	\end{align*} for each $t \in [0,T).$ Then $\sup\limits_{t \in [0,T]} \|I_3(t)\|_2 \lesssim \beta^{1/2} \sqrt{r+1} \|\Psi_1 - \Psi_2\|_{X^\beta}.$
	
	\vspace{.2cm}
	
	Therefore, the proof for contractivity goes as follows: \begin{align}
	\label{sme2} \|\Theta(\Psi_1) - \Theta(\Psi_2)\|_{_{X^\beta}}  &=\sup_{t \in [0,T]}\left\|\int_0^t  T(t-\sigma) \left[\cF(\Psi_1)-\cF(\Psi_2)\right]d\sigma\right\|_{\mathbb{H}_1}\nonumber \\
	&\leqslant \dfrac{C_{\omega_1}}{\tau} \sup_{t \in [0,T]} \tau\left\|\cF(\Psi_1)(t)-\cF(\Psi)(t)\right\|_{\mathbb{H}_1} \nonumber \\ &\lesssim \sup\limits_{t \in [0,T]} \left(\|I_1(t)\|_2 + \|I_2(t)\|_2 + \|I_3(t)\|_2\right) \nonumber\\ & \lesssim \left(2\beta + \beta^{1/2}\sqrt{r+1}\right) \|\Psi_1 - \Psi_2\|_{X^\beta} = C_{\beta} \|\Psi_1 - \Psi_2\|_{X^\beta}\end{align} owning the property $C_{\beta} < 1$ to the smallness of $\beta$.
\end{proof}

\
Notice that exponential stability of the linear problem in $\mathbb{H}$ and $\mathbb{H}_1$ allows we to obtain the estimates \eqref{invar}, \eqref{invar2} and \eqref{sme2} with right hand side time--independent. This allows us to take $T = \infty$ is all of them and repeat the same construction to obtain a fixed--point of $\Theta$ defined in the whole $\n{R}_+.$

This completes the proof of theorem \ref{nsem} by taking $\rho = \rho_\beta$.

\section{Proof of Theorem  \ref{expnonl}-Uniform Nonlinear Stability}\label{refext2}

In this section we show that one can easily show that the solution of the nonlinear problem decay exponentially to zero as $t\to \infty$ by taking advantage of three facts established in this paper. \begin{itemize}
	\item[\bf (i)] The fact that the solution is a fixed point of the map $\Theta$ defined in \eqref{varpar}, and therefore can be implicitly represented as \begin{equation}
	\label{easyv}\Phi(t) = T(t)\Phi_0 + \int_0^tT(t-\sigma)\cF(\Phi)(\sigma)d\sigma
	\end{equation}
	
	\item[\bf (ii)] The fact that our existence of global solution result requires smallness of initial data \emph{only} in the a lower topology and the use of this along with interpolation inequalities allowed us to obtain the key estimate \eqref{contcF}.
	
	\item[\bf (ii)] The fact that the semigroup $T(t)$ in \eqref{easyv} is uniformly exponentially stable in both $\mathbb{H}$ and $\mathbb{H}_1$.
\end{itemize}

The final result of this section is the following. \begin{theorem}\label{thmfinexp}
	There exists $\rho > 0$ such that the solution $\Phi$ constructed in \eqref{nsem} is such that \begin{equation}
	\label{expd} \|\Phi(t)\|_{\mathbb{H}_1} \leqslant 2M_1e^{-\frac{\omega_1}{2} t}\|\Phi_0\|_{\mathbb{H}_1}
	\end{equation}for all $t \geqslant 0$, where $M_1, \omega_1$ are the constants involved in the uniform stability of the linear semigroup $T(t).$
\end{theorem} The proof of this result relies heavily on the facts (i)--(ii) outlined above and a Grownwall type inequality. This inequality seems to have been originally introduced in \cite{beesack_CUDM_1975}, but here we are using \cite[Corollary 1, p. 389]{beesack_JMAA_1984}. We state the inequality here for convenience, but in a version which is suitable for our use in what follows. We invite the reader to consult \cite{beesack_CUDM_1975,beesack_JMAA_1984} and references therein for more details \begin{lemma}[Grownwall--Beesack Inequality] 
  Let $u,f,g,h:\BR \to \BR$ measurable functions such that $fh,gh$ and $uh$ are integrable. If $u,f,g,h$ are nonnegative and \begin{equation}
\label{u1} u(t) \leqslant f(t) + g(t)\int_0^t h(\sigma)u(\sigma)d\sigma
\end{equation} then \begin{equation}
\label{u2} u(t) \leqslant f(t) + g(t) \int_0^t f(\sigma)h(\sigma) \exp\left\{\int_\sigma^t g(s)h(s)ds\right\}d\sigma
\end{equation}
\end{lemma}

\begin{proof}[\bf Proof of Theorem \ref{thmfinexp}.] We use the same constants as in \eqref{exp}, that is, we use that \begin{equation}\label{u3}\|T(t)\|_{\calL(\mathbb{H}_1)} \leqslant M_1e^{-\omega_1 t}\end{equation} for all $t.$ Moreover, we know that the solution $\Phi$ exists in some $X^\beta$ for $\beta > 0$ small and that the whole argument of the proof for the existence of global solution would still be true if one decreased $\beta.$ Therefore, by possibly taking it smaller, we assume \begin{equation}\label{betasmall}\beta = \beta(\tau) < \dfrac{\tau\omega_1}{2M_1 C}\end{equation} 
 where $\omega_1 = \omega_1(\tau)$ is the rate of exponential decay of the semigroup $T(t)$ in $\mathbb{H}_1$ for a fixed $\tau > 0.$ As in the proof of global wellposedness, we take $\rho = \rho_\beta.$ We then compute, via \eqref{u3} and \eqref{contcF}  \begin{align*}
\|\Phi(t)\|_{\mathbb{H}_1} &\leqslant \|T(t)\Phi_0\|_{\mathbb{H}_1} + \int_0^t\|T(t-\sigma)\cF(\Phi)(\sigma)\|_{\mathbb{H}_1}d\sigma \\ &\leqslant M_1e^{-\omega_1 t} \|\Phi_0\|_{\mathbb{H}_1} + \int_0^t M_1e^{-\omega_1 (t-\sigma)}\|\cF(\Phi)(\sigma)\|_{\mathbb{H}_1}d\sigma \\ &\leqslant M_1e^{-\omega_1 t} \|\Phi_0\|_{\mathbb{H}_1} + \dfrac{M_1C\beta}{\tau}e^{-\omega_1 t}\int_0^t e^{\omega_1 \sigma}\|\Phi(\sigma)\|_{\mathbb{H}_1}d\sigma.
\end{align*} We then apply the Grownwall--Beesack inequality with $$u(t) = \|\Phi(t)\|_{\mathbb{H}_1}, \qquad f(t) =  M_1e^{-\omega_1 t} \|\Phi_0\|_{\mathbb{H}_1}, \qquad g(t) = \dfrac{M_1C\beta}{\tau}e^{-\omega_1 t}, \qquad h(t) = e^{\omega_1 t}$$ to obtain \begin{align}
\|\Phi(t)\|_{\mathbb{H}_1} &\leqslant M_1e^{-\omega_1 t} \|\Phi_0\|_{\mathbb{H}_1} + \dfrac{M_1^2C\beta\|\Phi_0\|_{\mathbb{H}_1}}{\tau}e^{-\omega_1 t}\int_0^t \exp\left\{\dfrac{M_1 C \beta}{\tau}(t-\sigma)\right\}d\sigma \nonumber \\ &= M_1e^{-\omega_1 t} \|\Phi_0\|_{\mathbb{H}_1} + M_1\|\Phi_0\|_{\mathbb{H}_1}
\exp\left\{\left(\dfrac{M_1 C \beta}{\tau}-\omega_1\right)t\right\}
\left(1-\exp\left\{-\dfrac{M_1 C \beta}{\tau}t\right\}\right) 
\nonumber \\ &\leqslant M_1e^{-\omega_1 t} \|\Phi_0\|_{\mathbb{H}_1} + M_1\|\Phi_0\|_{\mathbb{H}_1}
\exp\left\{\left(\dfrac{M_1 C \beta}{\tau}-\omega_1\right)t\right\} \leqslant 2M_1e^{-\frac{\omega_1}{2} t} \|\Phi_0\|_{\mathbb{H}_1}, \label{last}
\end{align} and we observe that due to \eqref{betasmall} we have $$\dfrac{M_1 C \beta}{\tau}-\omega_1 < -\dfrac{\omega_1}{2} <  0.$$ The proof is complete.

\end{proof}

\begin{corollary}\label{corexp} With reference to Section 5, let $\beta_0$ be the largest  number such that the map $\Theta$ has a fixed point in $X^{\beta_0}$ which is, moreover, uniformly exponentially stable as in Theorem \ref{expnonl}. Let $\omega: (0,\beta_0] \to \BR_+$ be the function that maps each $\beta>0$ to the decay rate $\omega(\beta)$. Then there exists another function $\underline \omega: (0,\beta_0] \to \BR_+$ such that $\omega(\beta) \geqslant \underline\omega(\beta)$ for all feasible $\beta$ and \begin{equation}
	\label{limexp} \lim\limits_{\beta \to 0} \underline\omega(\beta) = \omega_1,
	\end{equation} where $\omega_1$ is the decay rate of the linear semigroup $T(t).$
\end{corollary}
\begin{proof}[\bf Proof of Corollary \ref{corexp}] The proof of Theorem \ref{expnonl} already provides a proof of Corollary \ref{corexp}. Indeed, it suffices to define $\underline\omega:(0,\beta_0] \to \BR_+$ by $$\underline\omega(\beta) = \omega_1 - \dfrac{M_1 C \beta}{\tau}>0.$$
\end{proof}

\ifdefined\xxxxxxx

\section{Nonlinear Stabilization -- Hard Version}\label{refext}

Since global solutions for the nonlinear problem were constructed, nonlinear stability follows from a careful improvement of linear energy estimates. The most critical ingredient for this section the fact that lower order terms can be controlled by the damping of the system. We state this as a Lemma here, and the (technical) proof will be provided in the next section. \begin{proposition}\label{compuniq2}
	For $T>0$ there exists a constant $ C_T > 0 $ such that  the following inequality holds:
	\begin{equation}\label{abslo}
	lot_\delta(z)
	\leq C_T\int_0^T D_{\Psi}(s)ds 
	\end{equation}
\end{proposition} 

As a consequence of inequality \eqref{abslo}, uniform stability is established as we explain below.

\

As in the linear case, the first step for nonlinear stabilization is to study the properties of $E_1.$ Due to nonlinear terms being controlled only by higher energy, one cannot, in general, achieve exponential stability of $E_1$ alone, as was done in linear case. However,it paves the path for later higher level estimates. \begin{lemma} There exists a constant $C_T>0$ such that \begin{align}
	\label{enin} \int_0^T E_1(t)dt &\lesssim E_1(T) + \int_0^T E^{1/4}\E^{5/4}(t)dt\nonumber \\ &+ C_T\left[\int_0^T E^{1/2}(t)\E^{3/2}(t)dt + \int_0^T D_\Phi(t)dt \right].
	\end{align}
\end{lemma}\begin{proof}
	First, recall that for $f \in L^1(\BR_+,L^2(\Omega))$ we can modify identity \eqref{e1id} to obtain, for all $0 \leqslant \sigma \leqslant t \leqslant T$
	\begin{equation}\label{e1id2}
	E_1(t) + \int_\sigma^t D_\Psi(s)ds
	= E_1(\sigma) + \int_\sigma^t\int_\Omega fz_t d\Omega ds.
	\end{equation} Second, recall the general estimate  \eqref{e1in1}
	\begin{align}\label{e1in12}
	\int_s^{T-s} E_1(t) dt 
	&\lesssim [E_1(s)+E_1(T-s)] + 
	\nonumber \\ &C_T\left[\int_0^T D_\Psi(s)ds
	+\int_Q f^2 dQ + lot_\delta(z)\right].\end{align} which holds for $0<s<T/2$. 
	
	By using \eqref{e1in12}, we can get a full $L^1(0,T;\BR)$--norm of $E_1$. Indeed, for $s \leqslant 1/2$ small and identity \eqref{e1id2} (for $\sigma = 0$) we have
	\begin{align}\label{auxaux}
	\left(\int_0^s + \int_{T-s}^T\right) E_1(t) dt &\lesssim E_1(0) + \int_Q |fz_t|dQ 
	\end{align} 
	Then adding the above inequality with \eqref{e1in12} along with \eqref{abslo} we obtain \begin{align}\label{new1}
	\int_0^T\!\!\! E_1(t)dt &\lesssim E_1(0) \!+ E_1(s) + \!E_1(T-s) +\! \int_Q \!|fz_t|dQ + C_{T}\left[\int_0^T\!\! D_\Phi(s)ds + \int_Q\! f^2dQ\right]
	\end{align}and then using \eqref{e1id2} -- with $t = T$ and $\sigma = 0, s$ and $T-s$ respectively -- we infer that \begin{align}
	\label{w1} E_1(0) + E_1(s) + E_1(T-s) &= 3E_1(T) + \left(\int_0^T + \int_s^T + \int_{T-s}^T\right)\left[D_\Phi(t) - \int_\Omega fz_td\Omega dt
	\right] \nonumber \\ & \lesssim E_1(T) + \int_0^T D_\Phi(s)ds + \int_Q |fz_t|dQ
	\end{align} which improves \eqref{new1} to \begin{align}\label{new2}
	\int_0^T\!\!\! E_1(t)dt &\lesssim E_1(T) +\! \int_Q \!|fz_t|dQ + C_{T}\left[\int_0^T\!\! D_\Phi(s)ds + \int_Q\! f^2dQ\right]
	\end{align}
	
	We now estimate the nonlinear terms in \eqref{new2} by making $f = u_t^2 + uu_{tt}.$ We have two estimates:\begin{align}
	\int_Q f^2dQ &= \int_0^T \|u_t^2 + uu_{tt}\|_2^2dt \lesssim \int_0^T\left(\|u_t^2\|_2^2 + \|uu_{tt}\|_2^2\right)dt \nonumber \\ &\lesssim \int_0^T \left(\|\nabla u_t\|_2^4 + \|\nabla u\|_2\|\Delta u\|_2\|u_{tt}\|_2^2\right)dt \nonumber \\ &\lesssim \int_0^T E^{1/2}(t)\E^{3/2}(t)dt. \label{estf}
	\end{align} where, to obtain second inequality above, we used the embedding $H_{\Gamma_1}^1(\Omega) \hookrightarrow L^4(\Omega)$ followed by the embedding $H_{\Gamma_1}^2(\Omega) \hookrightarrow L^\infty(\Omega)$ and the interpolation inequality $$\|u\|_\infty^2 \leqslant \|u\|_{H^1}\|u\|_{H^2}$$ valid for all $u \in H^2(\Omega).$ The other estimate is \begin{align}
	\int_Q |fz_t|dQ &\leqslant \int_0^T \|z_t\|\|u_t^2 + uu_{tt}\|_2dt \lesssim \int_0^T E^{1/4}(t)\E^{5/4}(t)dt \label{estf2}
	\end{align}where we have used \eqref{estf} along with $\|z_t\|_2 \leqslant \E^{1/2}(t)$ for all $t.$ Then \eqref{new2} becomes \begin{align}
	\label{new3} \int_0^T E_1(t)dt &\lesssim E_1(T) + \int_0^T E^{1/4}\E^{5/4}(t)dt\nonumber \\ &+ C_T\left[\int_0^T E^{1/2}(t)\E^{3/2}(t)dt + \int_0^T D_\Phi(t)dt \right],
	\end{align} then \eqref{enin} follows. \end{proof}
	
	\begin{lemma}
		The inequality \begin{equation}
		\label{TINE} TE_1(T) \leqslant \int_0^T E_1(t)dt + C_T\int_0^T E^{1/4}(t)\E^{5/4}(t)dt
		\end{equation} holds for all $T>0$ arbitrary, but fixed.
	\end{lemma}\begin{proof}
	It follows from integrating \eqref{e1id2} in time on $(0,T).$
\end{proof}

	\begin{proposition}
		There exists $\eta = \eta(T) < 1$ and $M = M(\eta), \omega = \omega(\eta) > 0$ such that for $t \in [(m-1)T,mT)$ we have \begin{align}
		\label{stabmts} E_1(t) \leqslant M e^{-\omega t} E_1(0) + \sum\limits_{k=1}^m \eta^k\int_{(m-k)T}^{(m-k+1)T}\Upsilon(\sigma)d\sigma + \int_{mT}^{(m+1)T}\Upsilon(\sigma)d\sigma
		\end{align} where  \begin{equation}
		\label{up0s} \Upsilon(t) := E^{1/4}\E^{5/4}(t)+ E^{1/2}(t)\E^{3/2}(t).
		\end{equation} 
	\end{proposition}\begin{proof}
	It follows from \eqref{TINE} along with \eqref{enin} that \begin{align}
	\label{enin5} TE_1(T) + \int_0^T E_1(t)dt &\leqslant CE_1(T) + C\int_0^T E^{1/4}\E^{5/4}(t)dt\nonumber \\ &+ CC_T\left[\int_0^T E^{1/2}(t)\E^{3/2}(t)dt + \int_0^T D_\Phi(t)dt \right],
	\end{align} where the constant $C$ here appears because we changed $\lesssim$ to $\leqslant.$ We then estimate the damping term as \begin{equation}
	\label{damp1} \int_0^T D_\Phi(s)ds \leqslant E_1(0)-E_1(T) + C\int_0^TE^{1/4}(t)\E^{5/4}(t)dt
	\end{equation} then, we a possible bigger $C$, we rewrite \eqref{enin5} as
	
	\begin{align}
	\label{enin6} (T-C+CC_T)E_1(T) + \int_0^T E_1(t)dt &\leqslant CC_TE_1(0) + CC_T\int_0^T E^{1/4}\E^{5/4}(t)dt\nonumber \\ &+ CC_T\int_0^T E^{1/2}(t)\E^{3/2}(t)dt.
	\end{align}This implies \begin{equation}
	\label{stab0} E_1(T) \leqslant \eta E_1(0) + \eta \int_0^T \Upsilon(\sigma)d\sigma,
	\end{equation}  where $$\eta := \dfrac{CC_T}{T-C+CC_T} < 1$$ for a fixed large $T$ and $\Upsilon$ collects the nonlinear terms, i.e., \begin{equation}
	\label{up0} \Upsilon(t) := E^{1/4}\E^{5/4}(t)+ E^{1/2}(t)\E^{3/2}(t).
	\end{equation} Now we notice that \eqref{stab0} was the result of the analysis made within the interval of fixed length $T>0$. We now iterate the same process for the interval (again of length $T$) $[(m-1)T,mT]$ for $m = 1,2,...$ and this gives \begin{equation}
	\label{stabm} E_1(mT) \leqslant \eta^m E_1(0) + \sum\limits_{k=1}^m \eta^k\int_{(m-k)T}^{(m-k+1)T}\Upsilon(\sigma)d\sigma
	\end{equation} Therefore, given any $t \in \BR_+$, by writing $t = mT + s$ for some $m \in \mathbb{N}$ and $s \in [0,T)$ and using \eqref{e1id2} we have \begin{align}
	\label{stabmt} E_1(t) & \leqslant  E_1(mT) + \int_{mT}^{t}\Upsilon(\sigma)d\sigma \nonumber \\ & \leqslant \eta^m E_1(0) + \sum\limits_{k=1}^m \eta^k\int_{(m-k)T}^{(m-k+1)T}\Upsilon(\sigma)d\sigma + \int_{mT}^{t}\Upsilon(\sigma)d\sigma \nonumber \\ &\leqslant M e^{-\omega t} E_1(0) + \sum\limits_{k=1}^m \eta^k\int_{(m-k)T}^{(m-k+1)T}\Upsilon(\sigma)d\sigma + \int_{mT}^{t}\Upsilon(\sigma)d\sigma
	\end{align} where $M = \eta^{-1}$ and $\omega := -T^{-1}\ln(\eta)$. \end{proof} \begin{remark}
	We notice that \eqref{stabmts} \emph{almost} yields exponential decay of $E_1$. The nonlinear remainder can then be taken care of by higher energy estimates.
\end{remark}

The final result of this section is the following. \begin{theorem}
	There exist $\overline \omega, \overline M> 0$ such that the inequality \begin{equation}
	\label{nonlinearstab} \E(t) \leqslant \overline Me^{-\overline\omega t} \E(0)
	\end{equation} for all $t \geqslant 0.$ 
\end{theorem}\begin{proof}
Denote by $\Upsilon_m$ the quantity \begin{equation}
\label{up5} \Upsilon_m := \sum\limits_{k=1}^m \eta^k\int_{(m-k)T}^{(m-k+1)T}\Upsilon(\sigma)d\sigma + \int_{mT}^{(m+1)T}\Upsilon(\sigma)d\sigma,
\end{equation} to ease readability. Then we start with the knowledge that \begin{equation}
\label{k1} E_1(t) \leqslant Me^{-\omega t}E_1(0) + \Upsilon_m
\end{equation} and recall that $m$ is connect to $t$ via $t = mT + s$ for some $s \in [0,T)$ with $T$ large enough so \eqref{stabmts} holds. Now recall that, via the relation $bz = bu_t + u$ we can write \begin{equation}\label{k2} u(t) = e^{-\frac{c^2}{b}t}u_0 + \int_0^te^{-\frac{c^2}{b}(t-\tau)}z(\tau)d\tau,\end{equation} and this allow both terms in $E_0$ (see \eqref{E0}) to be bounded above due to current \emph{almost} exponential bounds for $E_1.$ More precisely, there exist $C_T, M_1, \omega_1 >0$ such that \begin{equation}
\label{k3} E(t) \leqslant M_1e^{-\omega_1 t}E(0) + C_T\Upsilon_m.
\end{equation} In order to extend this to the higher energy $\E$ we again use the relation $bz = bu_t + c^2u$ to write \begin{equation}
\label{k4} \Delta u(t) = e^{-\frac{c^2}{b}t}\Delta u_0 + \int_0^te^{-\frac{c^2}{b}(t-\tau)}\Delta z(\tau)d\tau,
\end{equation} and then here we use the orignal PDE (in $z$) and perform computations like in \eqref{ode} \begin{align}\label{ode2}\Delta u(t) &= e^{-\frac{c^2}{b}t}\Delta u_0 +\int_0^t e^{-\frac{c^2}{b}(t-\sigma)} \Delta z(\sigma)d\sigma \nonumber \\ &= e^{-\frac{c^2}{b}t}\Delta u_0 +\dfrac{\tau}{b}\int_0^t e^{-\frac{c^2}{b}(t+\sigma)}\left[z_{tt}(\sigma)+\gamma u_{tt}(\sigma)\right]d\sigma \nonumber \\ &= e^{-\frac{c^2}{b}t}\Delta u_0 +\dfrac{\tau}{b} \left[z_t(t)+\gamma u_t(t) - e^{-\frac{c^2}{b}t}[z_t(0)+\gamma u_1] \right] \nonumber \\ &+ \dfrac{c^2}{b^2}\int_0^t e^{-\frac{c^2}{b}(t-\sigma)}[z_t(\sigma)+\gamma u_t(\sigma)]d\sigma \in L^2(\Omega). \end{align} from where it follows that there exist $C_T, M_2, \omega_2 > 0$ such that \begin{equation}\label{k5}\E(t) + \int_0^t\E(s)ds \leqslant M_2e^{-\omega_2 t}\E(0) + C_T \Upsilon_m.\end{equation} And as a last step we absorb $\Upsilon_m.$ 	To this end, recall that we started with $\mathbb{H}_1$--initial datum which is $\mathbb{H}$--small, we have $E(0) \leqslant \rho$ and $\rho$ can be chosen as small as needed. Moreover, we showed in the previous section that solutions would remain $\mathbb{H}$--small and $\mathbb{H}_1$--bounded for all times, that is, $E(t) \lesssim \rho$ while $\E(t) \leqslant R$ for all times. This means that we can estimate $\Upsilon$ as follows \begin{equation}
\label{up2}  \Upsilon(t) \leqslant \left[(\rho R)^{1/4} + (\rho R)^{1/2}\right]\E(t) := C_\rho \E(t)
\end{equation} therefore we estimate $\Upsilon_m$ as follows \begin{align}
\label{up6} \Upsilon_m &= \sum\limits_{k=1}^m \eta^k\int_{(m-k)T}^{(m-k+1)T}\Upsilon(\sigma)d\sigma + \int_{mT}^{(m+1)T}\Upsilon(\sigma)d\sigma \nonumber \\ & \leqslant \sum\limits_{k=1}^m \eta^k\int_{(m-k)T}^{(m-k+1)T} C_\rho \E(\sigma)d\sigma + \int_{mT}^{(m+1)T} C_\rho \E(\sigma)d\sigma \nonumber \\ & \leqslant C_\rho\left(\sum\limits_{k=0}^m \eta^k\right)  \int_0^t\E(\sigma)d\sigma \leqslant \dfrac{C_\rho}{1-\eta}\int_0^t \E(\sigma)d\sigma
\end{align}then we can improve \eqref{k5} to \begin{equation}\label{k6}\E(t) + \left(1-\dfrac{C_\rho}{1-\eta}\right)\int_0^t\E(s)ds \leqslant M_2e^{-\omega_2 t}\E(0),\end{equation} which finishes the proof by choosing $\rho > 0$ small.
\end{proof}
\fi

\bibliographystyle{abbrvurl} 
\bibliography{ref2.bib}
\end{document}